\documentclass[10pt, a4paper, twoside]{amsart}

\usepackage{fancyhdr}
\date{\today}

\usepackage[usenames, dvipsnames]{color}
\definecolor{darkblue}{rgb}{0.0, 0.0, 0.45}

\usepackage[colorlinks	= true,
			raiselinks	= true,
			linkcolor	= darkblue, 
			citecolor	= Mahogany,
			urlcolor	= ForestGreen,
			pdfauthor	= {Peyman Mohajerin Esfahani},
			pdftitle	= {PhD Thesis},
			pdfkeywords	= {},
			pdfsubject	= {PhD Thesis},
			plainpages	= false]{hyperref}

\usepackage{dsfont,amssymb,amsmath,subfigure, graphicx} 
\usepackage{enumitem}

\usepackage{amsfonts,dsfont,mathtools, mathrsfs,amsthm} 
\usepackage[amssymb, thickqspace]{SIunits}
\usepackage{algorithm,algorithmicx}

\allowdisplaybreaks

\newtheorem{Thm}{Theorem}[section]
\newtheorem{Cor}[Thm]{Corollary}
\newtheorem{Lem}[Thm]{Lemma}
\newtheorem{Pro}[Thm]{Proposition}
\newtheorem{Def}[Thm]{Definition}
\newtheorem{Assumption}[Thm]{Assumption}

\newtheorem{Remark}[Thm]{Remark}

\newcommand{\mn}{\wedge}
\newcommand{\mx}{\vee}

\newcommand{\Lra}{\Longrightarrow}

\newcommand{\ra}{\rightarrow}
\newcommand{\da}{\downarrow}
\newcommand{\ua}{\uparrow}

\newcommand{\Let}{\coloneqq}
\newcommand{\diff}{\mathrm{d}}
\newcommand{\tr}{^\intercal}
\newcommand{\ball}[1]{\mathrm{#1}}
\newcommand{\eps}{\varepsilon}
\newcommand{\ol}[1]{\overline{#1}}

\newcommand{\R}{\mathbb{R}}
\newcommand{\N}{\mathbb{N}}
\newcommand{\secref}[1]{\S\ref{#1}}
\newcommand{\Sup}[3]{\sup_{#2 \in [#1,#3]}}
\newcommand{\Inf}[3]{\inf_{#2 \in [#1,#3]}}

\newcommand{\PP}{\mathds{P}}
\newcommand{\EE}{\mathds{E}}
\newcommand{\sigalg}{\mathcal{F}}
\newcommand{\Filtration}{\mathds{F}}
\newcommand{\E}[1]{\mathds{E}{\left[ #1 \right]} }
\newcommand{\traj}[4]{#1_{#4}^{#2;#3}}

\newcommand{\ind}[1]{\mathds{1}_{#1}}
\newcommand{\el}[1]{\ell \big( #1 \big)}

\newcommand{\RA}{\mathrm{RA}}

\newcommand{\setofst}[1]{\mathcal{T}_{[#1]}}
\newcommand{\controlset}{\mathbb}
\newcommand{\set}{\mathbb}
\newcommand{\controlmaps}[1]{\mathcal{#1}}
\newcommand{\control}[1]{\boldsymbol{#1}}

\newcommand{\process}[2]{(#1_{#2})_{#2 \ge 0}}
\newcommand{\Borelsigalg}[1]{\mathfrak{B}(#1)}
\newcommand{\borel}{\mathfrak{B}}

\DeclareMathOperator{\dist}{dist}
\newcommand{\wt}{\widetilde}
\newcommand{\wh}{\widehat}

\title{The Stochastic Reach-Avoid Problem and Set Characterization for Diffusions}
\thanks{PME and JL are with the Automatic Control Laboratory, ETH Z\"urich, 8092 Z\"urich, Switzerland; DC is with the Systems \& Control Engineering, IIT-Bombay, Powai, Mumbai 400076, India. Emails: {\{mohajerin,lygeros\}@control.ee.ethz.ch, chatterjee@sc.iitb.ac.in}}

\author{Peyman Mohajerin Esfahani, Debasish Chatterjee, and John Lygeros}

\begin{document} 
	\maketitle
\begin{abstract}                         
	In this article we approach a class of stochastic reachability problems with state constraints from an optimal control perspective. Preceding approaches to solving these reachability problems are either confined to the deterministic setting or address almost-sure stochastic requirements. In contrast, we propose a methodology to tackle problems with less stringent requirements than almost sure. To this end, we first establish a connection between two distinct stochastic reach-avoid problems and three classes of stochastic optimal control problems involving discontinuous payoff functions. Subsequently, we focus on solutions of one of the classes of stochastic optimal control problems---the exit-time problem, which solves both the two reach-avoid problems mentioned above. We then derive a weak version of a dynamic programming principle (DPP) for the corresponding value function; in this direction our contribution compared to the existing literature is to develop techniques that admit discontinuous payoff functions. Moreover, based on our DPP, we provide an alternative characterization of the value function as a solution of a partial differential equation in the sense of discontinuous viscosity solutions, along with boundary conditions both in Dirichlet and viscosity senses. Theoretical justifications are also discussed to pave the way for deployment of off-the-shelf PDE solvers for numerical computations. Finally, we validate the performance of the proposed framework on the stochastic Zermelo navigation problem. 
\end{abstract}

\section{Introduction}
\label{Section Introduction}
Reachability is a fundamental concept in the study of dynamical systems, and in view of applications of this concept ranging from engineering, manufacturing, biology, and economics, to name but a few, has been studied extensively in the control theory literature. One particular problem that has turned out to be of fundamental importance in engineering is the so-called ``reach-avoid'' problem. 

In the deterministic setting this problem deals with the determination of the set of initial states for which one can find at least one control strategy to steer the system to a target set while avoiding certain obstacles. This problem finds applications in, for example, air traffic management \cite{ref:TomlinLygerosSastry-2000} and security of power networks \cite{ref:Mohajerin_CDC10}. 

The set representing the solution of this problem is known as a capture basin \cite{ref:Aubin-1991}. A direct approach to compute the capture basin is formulated in the language of viability theory in \cite{ref:Cardaliaguet-1996, ref:CardaliaguetQuincampoixSaint-2002}. An alternative and indirect approach to reachability problems proceeds via level set methods defined by value functions that are solutions of appropriate optimal control problems. Employing dynamic programming techniques for reachability and viability problems, one can in turn characterize these value functions by solutions of the standard Hamilton-Jacobi-Bellman (HJB) equations corresponding to these optimal control problems \cite{ref:Lygeros-2004}. The focus of this article is on the stochastic counterpart of this problem. 

\subsection{The literature in the stochastic setting}

In the literature, probabilistic analogs of reachability problems have mainly been studied from an almost-sure perspective. For example, stochastic viability and controlled invariance are treated in \cite{ref:AubinPrato-1990, ref:AubinPratoFrankowska-2000, ref:BardiJensen-2002}. Methods involving stochastic contingent sets \cite{ref:AubinPrato-1998, ref:AubinPratoFrankowska-2000}, viscosity solutions of second-order partial differential equations \cite{ref:BuckdahnPengQuincampoixRainer-1998,ref:BardiGoatin-1999, ref:BardiJensen-2002}, derivatives of the distance function \cite{ref:PratoFrankowska-2001}, and equivalence relation to certain deterministic control systems \cite{ref:PratoFrankowska-2004} were all developed in this context. 

Geared towards similar almost-sure reachability objective, the article \cite{ref:SonerTouzi-02} introduced a new class of the so-called stochastic target problems, and characterized the solution via a dynamic programming approach. The differential properties of the almost-sure reachable set were also studied based on the geometrical partial differential equation which is the analogue of the HJB equation \cite{ref:SonerTouzi-geometric-02} in that setting.

Although almost sure versions of reachability specifications are interesting in their own right, they may be a too strict concept in some applications, particularly when a common specification is only to control the probability that undesirable events take place. In this regard, the authors of \cite{ref:BouchEliTouzi-10} recently extended the stochastic target framework of \cite{ref:SonerTouzi-02} to allow for unbounded control set, which together with the martingale representation theory, addresses the aforementioned almost-sure limitation in an augmented state space; see also the recent book \cite{ref:Touzi-13}. This article approaches the same question, but indirectly and from an optimal control perspective.

\subsection{Our methodology and contributions}
The stochastic ``reach-avoid" problems studied in this article are as follows:

	\begin{quote}
		$\RA$: \textsl{Given an initial state $x\in\R^n$, a horizon $T > 0$, a number $p\in\:[0, 1]$, and two disjoint sets $A, B\subset\R^n$, determine whether there exists a control policy such that the process reaches $A$ prior to entering $B$ within the interval $[0, T]$ with probability at least $p$.}
	\end{quote}

Observe that this is a significantly different problem compared to its almost-sure counterpart referred to above. It is of course immediate that the solution of the above problem is trivial if the initial state is either in $B$ (in which case it is almost surely impossible) or in $A$ (in which case there is nothing to do). However, for generic initial conditions in $\R^n\setminus(A \cup B)$, due to the inherent probabilistic nature of the dynamics, the problem of selecting a policy and determining the probability with which the controlled process reaches the set $A$ prior to hitting $B$ is non-trivial. In addition, we address the following slightly different reach-avoid problem compared to $\RA$ above, that requires the process to be in the set $A$ at time $T$:

	\begin{quote}
		$\widetilde{\RA}$: \textsl{Given an initial state $x\in\R^n$, a horizon $T > 0$, a number $p\in\:[0, 1]$, and two disjoint sets $A, B\subset\R^n$, determine whether there exists a policy such that with probability at least $p$ the controlled process resides in $A$ at time $T$ while avoiding $B$ on the interval $[0, T]$.}
	\end{quote}

Our methodology and contributions toward the above problems are summarized below:

\begin{enumerate}[label=(\roman*), itemsep = 0	mm, widest=iii]
\item We establish a link from the problems $\RA$ and $\wt \RA$ to three different classes of stochastic optimal control problems involving \emph{discontinuous} payoff functions in \secref{sec:connection}; 

\item focusing on the class of exit-time problems that addressed both the reach-avoid problems alluded above, we propose a \emph{weak} dynamic programming principle (DPP) leading to a (\emph{discontinuous}) PDE characterization along with appropriate boundary conditions; 

\item finally, in \secref{sec:connection 3 to 4} we provide theoretical justification that pave the analytical ground to deploy existing (\emph{continuous}) off-the-shelf PDE solvers for our numerical purposes. 
\end{enumerate}

More specifically, we first show that the desired set of initial conditions for the reach-avoid problems $\RA$ and $\wt \RA$ can be translated as super level sets of particular functions described in the context of stochastic optimal control problems (Propositions~\ref{prop:Reach-V1} and \ref{prop:Reach-Terminal}). Different classes of optimal control problems are suggested for each of the two reach-avoid problems, and it turns out that the class of exit-time problems with discontinuous payoff functions can adequately address both the reach-avoid problems. This connection is relatively straightforward and does not require any assumption on the underlying dynamics. We, however, are not aware of any results in the literature reflecting this connection. 

The exit-time problem with a continuous payoff function is a classical stochastic optimal control problem whose alternative PDE characterizations have been established in the literature; see for instance \cite[Section IV.7]{SonerBook}. However, these results are not directly applicable to our reach-avoid problems due to the discontinuity of the payoff function. We address this technical issue by developing a DPP in a weak sense in the spirit of \cite{BouchardTouzi_WeakDPP} (Theorem~\ref{Theorem DPP 1}). We emphasize that the results of \cite{BouchardTouzi_WeakDPP} were developed in the framework of fixed time horizon and the optimal stopping time. Neither of these settings is applicable to the exit-time problem. To that end, it turns out that we require some technical continuity properties which are essential for the proposed weak DPP (Proposition~\ref{prop:J1 semicontinuous}) as well as the respective boundary conditions (Proposition~\ref{prop:unif cont}). To the best of our knowledge, these continuity results are also new in the literature. It is also worth noting that this weak formulation avoids delicate issues related to a measurable selection in the context of optimal control problems.


Based on the proposed DPP, we characterize the value function as the (discontinuous) viscosity solution of a PDE (Theorem \ref{Theorem DPE 1}) along with boundary conditions in both viscosity and Dirichlet (pointwise) senses (Theorem \ref{thm:boundary}). We remark that due to the discontinuity of the payoff function, the viscosity boundary conditions involves a non-trivial regularity condition which is a stronger version of the requirement for the proposed DPP (see Proposition \ref{prop:unif cont}). These technical details are required to rigorously settle the PDE characterization for a stochastic exit-problem problem and we cannot find them elsewhere in the existing literature. 

Finally, we provide theoretical justifications (Theorem \ref{thm:approx}) so that the Reach-Avoid problem is amenable to numerical solutions by means of off-the-shelf PDE solvers, which have been mainly developed for continuous solutions. Preliminary results of this study were reported in \cite{ref:MohChaLyg-11} without covering the technical details and mathematical proofs. 

Organization of the article: In \secref{sec:ProblemStatement} we formally introduce the stochastic reach-avoid problems $\RA$ and $\wt \RA$ above. In \secref{sec:connection} we characterize the set of initial conditions that solve the reach-avoid problems in terms of super level sets of three different value functions. Focusing on the class of exit-time problems, in \secref{sec:Alternative Characterization} we establish a DPP and characterize it as the solution of a PDE along with some boundary conditions. Finally, \secref{sec:connection 3 to 4} presents results connecting those in \secref{sec:connection} and \secref{sec:Alternative Characterization} and justifies the deployment of the existing PDE solvers for numerical purposes. To illustrate the performance of our technique, the theoretical results developed in preceding sections are applied to solve the stochastic Zermelo navigation problem in \secref{sec:simulation}. We conclude with some remarks and directions for future work in \secref{sec:conclusion}. For better readability, some of the technical proofs are given in appendices. 
	
\paragraph{\bf Notation}
Given $a,b \in \R$, we define $a\mn b \Let \min\{a, b\}$ and $a\mx b \Let \max\{a, b\}$. We denote by $A^c$ (resp.\ $A^\circ$) the complement (resp.\ interior) of the set $A$. We also denote by $\ol{A}$ (resp.\ $\partial A$) the closure (resp.\ boundary) of $A$. We let $\ball{B}_r(x)$ be an open Euclidean ball centered at $x$ with radius $r$. The Borel $\sigma$-algebra on a topological space $\set{A}$ is denoted by $\borel(\set{A})$, and measurability on $\R^d$ will always refer to Borel-measurability. The indicator function $\ind{\set A}$ is defined through $\ind{A}(x) = 1$ if $x \in \set A$; $=0$ otherwise. Given function \(f:\set{A}\ra\R\), the lower and upper semicontinuous envelopes of \(f\) are defined, respectively, by $f_{*}(x) := \liminf_{x' \ra x} f(x')$ and $f^{*}(x) := \limsup_{x' \ra x} f(x')$. The set $\text{USC}(\set{A})$ (resp.\ $\text{LSC}(\set{A})$) denotes the collection of all upper semicontinuous (resp.\ lower semicontinuous) functions from $\set{A}$ to $\R$. Throughout this article all (in)equalities between random variables are understood in almost sure sense. For the ease of the reader, we also provide here a partial notation list which will be also explained in more details later throughout the article:

	\begin{enumerate}[label=$\bullet$, itemsep = 0mm, leftmargin=*, align=right, widest=iii] 
		\item $\set S \Let [0,T]\times \R^n$;
		\item $\controlmaps{U}_\tau$: set of $\Filtration_{\tau}$-progressively measurable maps into $\mathds{U}$;
		\item $\setofst{\tau_1,\tau_2}$ : the collection of all $\Filtration_{\tau_1}$-stopping times $\tau$ satisfying $\tau_1 \leq \tau \leq \tau_2$ $\PP$-a.s.\;
		\item $(\traj{X}{t,x}{\control{u}}{s})_{s \ge 0}$: stochastic process under the control policy $\control{u}$ and assumption $\traj{X}{t,x}{\control{u}}{s} \Let x$ for all $s \le t$;
		\item $\tau_A$: first entry time to $A$, see Definition \ref{def:entry time};
		\item $\mathcal{L}^u$: Dynkin operator, see Definition \ref{def:Dynkin operator}.
	\end{enumerate}

\section{The Setting and Statement of Problem}
\label{sec:ProblemStatement}
Consider a filtered probability space $(\Omega, \sigalg, \Filtration,\PP)$ whose filtration $\Filtration = (\sigalg_s)_{s\geq 0}$ is generated by an $n$-dimensional Brownian motion $\process{W}{s}$ adapted to $\Filtration$. Let the natural filtration of the Brownian motion $\process{W}{s}$ be enlarged by its right-continuous completion; --- the usual conditions of completeness and right continuity, where $\process{W}{s}$ is a Brownian motion with respect to $\Filtration$ \cite[p.\ 48]{ref:KarShr-91}. For every $t\ge0$, we introduce an auxiliary subfiltration $\Filtration_t \Let (\sigalg_{t,s})_{s\ge0}$, where $\sigalg_{t,s}$ is the $\PP$-completion of $\sigma\big(W_{r \mx t} - W_t, r \in [0,s]\big)$. Note that for $s \le t$, $\sigalg_{t,s}$ is the trivial $\sigma-$algebra, and any $\sigalg_{t,s}$-random variable is independent of $\sigalg_t$. By definitions, it is obvious that $\sigalg_{t,s} \subseteq \sigalg_s$ with equality in case of $t=0$.

Let $\controlset{U}\subset\R^m$ be a control set, and $\controlmaps{U}_t$ denote the set of $\Filtration_t$-progressively measurable maps into $\controlset{U}$.\footnote{Recall \cite[p.\ 4]{ref:KarShr-91} that a $\controlset{U}$-valued process $\process{y}{s}$ is $\Filtration_t$-progressively measurable if for each $T > 0$ the function $\Omega\times[0, T]\ni (\omega, s)\mapsto y(\omega, s)\in \controlset{U}$ is measurable, where $\Omega\times[0, T]$ is equipped with $\sigalg_{t,T}\otimes\Borelsigalg{[0, T]}$, $\controlset{U}$ is equipped with $\Borelsigalg{\controlset{U}}$, and $\Borelsigalg{S}$ denotes the Borel $\sigma$-algebra on a topological space $S$.} We employ the shorthand $\controlmaps{U}$ instead of $\controlmaps{U}_0$ for the set of all $\Filtration$-progressively measurable policies.  We also denote by $\mathcal{T}$ the collection of all $\Filtration$-stopping times. For $\tau_1, \tau_2 \in \mathcal{T}$ with $\tau_1 \le \tau_2$ $\PP$-a.s., the subset $\setofst{\tau_1,\tau_2}$ is the collection of all $\Filtration_{\tau_1}$-stopping times $\tau$ such that $\tau_1 \leq \tau \leq \tau_2$ with probability 1. Note that all $\Filtration_{\tau}$-stopping times and $\Filtration_{\tau}$-progressively measurable processes are independent of $\sigalg_{\tau}$.		

The basic object of our study concerns the $\R^n$-valued stochastic differential equation (SDE)
	\begin{equation}
		\label{SDE}
		\diff X_s = f(X_s, u_s)\,\diff s + \sigma(X_s, u_s)\,\diff W_s,\qquad X_0 = x,\quad  s \ge 0,
	\end{equation}
where $f:\R^n\times \controlset{U} \ra \R^n$ and $\sigma:\R^n\times\controlset{U} \ra\R^{n \times d}$ are continuous and Lipschitz in first argument uniformly with respect to the second argument, $\process{W}{s}$ is the above standard $d$-dimensional Brownian motion, and the control set $\set U \subset \R^m$ is compact.\footnote{We slightly abuse notation and earlier used $\sigma$ as a sigma algebra as well. However, it will be always clear from the context to which $\sigma$ we refer.} It is known that under this setting the SDE \eqref{SDE} admits a unique strong solution \cite{ref:Borkar-05}. We let $(\traj{X}{t, x}{\control{u}}{s})_{s \ge t}$ denote the unique strong solution of \eqref{SDE} starting from time $t$ at the state $x$ under the control $\control{u}$. For future notational simplicity, we slightly generalize the definition of $\traj{X}{t, x}{\control{u}}{s}$, and extend it to the whole interval $[0,T]$ where $\traj{X}{t, x}{\control{u}}{s} \Let x$ for all $s$ in $ [0, t]$. 
%
	
Given an initial time $t$ and the disjoint sets $A, B \subset \R^n$, we are interested in the set of initial conditions $x \in \R^n$ where there exists an admissible control $\control{u} \in \controlmaps{U}$ such that with probability more than $p$ the state trajectory $\traj{X}{t,x}{\control{u}}{s}$ hits the set $A$ before set $B$ within the time horizon $T$. Our main objective in this article is to propose a framework in order to characterize this set of initial condition, which is formally introduced as follows.

	\begin{Def}[Reach-Avoid within \({[0, T]}\)]
	\label{def:RA within}
	\begin{flalign*}
	    \RA (t,p;A,B) \Let \Big\{x \in \R^n& ~\big |~ \exists \control{u} \in \controlmaps{U} ~:~ \\
	    &\PP \Big( \exists s \in [t,T], ~ \traj{X}{t,x}{\control{u}}{s} \in A ~ \text{and} ~ \forall r \in [t,s] ~ \traj{X}{t,x}{\control{u}}{r} \notin B \Big) > p \Big \}.
	\end{flalign*}
	\end{Def}

We also study another reach-avoid problem denoted by $\widetilde{\RA}$ as mentioned in \secref{Section Introduction}. As opposed to Definition \ref{def:RA within} that only requires to reach the target sometime within the interval $[t,T]$, the problem $\wt{\RA}$ poses constraint for being in the target set at time $T$ while avoiding barriers over the period $[t, T]$. Namely, we define the set $\widetilde{\RA}(t,p;A,B)$ as the set of all initial conditions for which there exists an admissible control strategy $\control{u} \in \controlmaps{U}$ such that with probability more than $p$, $\traj{X}{t,x}{\control{u}}{T}$ belongs to $A$ and the process avoids the set $B$ over the interval $[t, T]$.

	\begin{Def}[Reach-Avoid at the terminal time $T$]
	    \label{def:RA terminal}
	    \begin{flalign*}
	    \widetilde{\RA}(t,p;A,B) \Let \Big\{ x \in \R^n& ~\big |~ \exists \control{u} \in \controlmaps{U} ~:~ \\
	    &\PP \Big( \traj{X}{t,x}{\control{u}}{T} \in A ~ \text{and} ~ \forall r \in [t,T] ~ \traj{X}{t,x}{\control{u}}{r} \notin B \Big) > p \Big \}.
		\end{flalign*}
	\end{Def}

\section{A Connection to Stochastic Optimal Control Problem}
\label{sec:connection}
In this section we establish a connection between the stochastic reach-avoid problems $\RA$ and $\wt{\RA}$ to three different classes of stochastic optimal control problems. The results presented in this section rely on pathwise analysis, and are not necessarily confined to the SDE setting. The following definition is one of the key elements in our framework. 
	
	\begin{Def}[First entry time]
	\label{def:entry time}
	Given a control $\control{u}$, the process $(\traj{X}{t, x}{\control{u}}{s})_{s \ge t}$, and a set $A \subset \R^n$, we introduce\footnote{By convention, $\inf\emptyset = \infty$.} the \emph{first entry time} to $A$ by
	\begin{flalign}
		\tau_A(t,x) = \inf \big\{ s \geq t ~|~ \traj{X}{t, x}{\control{u}}{s} \in A \big\}.
	\end{flalign}
	\end{Def}

Let us note that the first entry time in Definition \ref{def:entry time} is indeed an $\Filtration_t$-stopping time \cite[Theorem 1.6, Chapter 2]{ref:EthKur-86}. 

	\begin{Remark}[Entry time properties]
	\label{rem:entry time}
	In light of almost sure continuity of the solution process, for any initial condition $(t,x)$ and control $\control{u} \in \controlmaps{U}$ we have
	\begin{subequations}
		\begin{flalign}
			\label{rem:A cup B} \tau_{A \cup B} &= \tau_A \mn \tau_B, \\
			\label{rem:exit time A} \traj{X}{t,x}{\control{u}}{s} \in A &\Lra \tau_{A} \leq s, \\
			\label{rem:exit time closed set} A ~ \text{is closed} &\Lra \traj{X}{t,x}{\control{u}}{\tau_A } \in A.
		\end{flalign}
	\end{subequations}
	\end{Remark}
	One can think of several different ways of characterizing probabilistic reach-avoid sets,

	see for instance \cite{Chatterjee2011} and the references therein dealing with discrete-time problems. Motivated by these works, we consider value functions involving expectation of indicator functions of certain sets. Three alternative characterizations are considered and we show all three are equivalent. We define the functions $V_i : [0,T] \times \R^n \ra [0,1]$, $i\in \{1,2,3\}$, as
    \begin{subequations}
    \label{Vi}
    \begin{align}
    	\label{V1}	V_1(t,x)    & := \sup_{\control{u} \in \controlmaps{U}} \EE\bigl[\ind{A}( \traj{X}{t,x}{\control{u}}{\wh{\tau}})\bigr] \qquad \text{where} \quad \wh{\tau} := \tau_{A \cup B} \mn T,\\
    	\label{V2}	V_2(t,x)	& := \sup_{\control{u} \in \controlmaps{U}} \EE\biggl[\Sup{t}{s}{T} \Bigl\{ \ind{A}(\traj{X}{t,x}{\control{u}}{s}) \mn \Inf{t}{r}{s} \ind{B^c}(\traj{X}{t,x}{\control{u}}{r})\Bigr\} \biggr],\\
        \label{V3}	V_3(t,x)	& := \sup_{\control{u} \in \controlmaps{U}} \sup_{\tau \in \setofst{t,T}}\inf_{\sigma \in \setofst{t,\tau}} \E{\ind{A}(\traj{X}{t,x}{\control{u}}{\tau}) \mn \ind{B^c}(\traj{X}{t,x}{\control{u}}{\sigma})}.
    \end{align}
    \end{subequations}
Here $\tau_{A\cup B}$ is the entry time introduced in Definition \ref{def:entry time}, and depends on the initial condition $(t,x)$. For notational simplicity, we drop the initial condition in this section. 

In \eqref{V1}, the process $\traj{X}{t,x}{\control{u}}{\cdot}$ is controlled until a particular stopping time $\wh \tau$, by which instant the process either exits from the set $A\cup B$ or the terminal time $T$ is reached. In this light, the stochastic optimal control \eqref{V1} is also known as exit-time problem. A sample $\omega \in \Omega$ is a ``successful" path if the stopped process $\traj{X}{t,x}{\control{u}}{\wh \tau(\omega)}(\omega)$ resides in $A$. This requirement is captured via the payoff function $\ind{A}(\cdot)$. 

In the definition of $V_2$ in \eqref{V2}, there is no stopping time, and one may observe that the entire process $\traj{X}{t,x}{\control{u}}{\cdot}$ is considered. Here the requirement of reaching the target set $A$ before the avoid set $B$ is taken into account by the supremum and infimum operations and payoff functions $\ind{A}$ and $\ind{B^c}$. 

In a fashion similar to \eqref{V1}, the function $V_3$ in \eqref{V3} involves some stopping time strategies. The stopping strategies, however, are not fixed and the stochastic optimal control problem can be viewed as a game between two players with different authorities. Namely, the first player has both control $\control{u}\in \controlmaps{U}$ and stopping $\tau \in \setofst{t,T}$ strategies whereas the second player has only a stopping strategy $\sigma \in \setofst{t,\tau}$, which is dominated by the first player's stopping time $\tau$; each player contributes through different maps to the payoff function.

	\begin{Pro}[Connection from $\RA$ to \eqref{Vi}]
	\label{prop:Reach-V1}
		Let sets $A, B$ be disjoint closed subsets of $\R^n$. Then, the equality $V_1 = V_2 = V_3$ holds on $\set S \Let  [0,T]\times \R^n$, and we have
		\begin{equation*}
			\RA(t,p;A,B) =  \big\{ x \in \R^n ~|~ V_i(t,x) > p \big\}, \quad i \in \{1,2,3\},
		\end{equation*}
		where the set $\RA$ is the set defined in Definition \ref{def:RA within}.
	\end{Pro}
    
	\begin{proof}
		See \ref{app:A}.
	\end{proof}

One can establish a connection between the reach-avoid problem $\wt \RA$ in Definition \ref{def:RA terminal} and different classes of stochastic optimal control problems along lines similar to Propositions \ref{prop:Reach-V1}. To this end, let us define the value functions $\wt{V}_i : [0,T] \times \R^n \ra [0,1]$, $i\in\{1,2,3\}$, as
	\begin{subequations}
	\label{Vtilde}
	\begin{align}
		\label{Vtilde 1}
	\widetilde{V}_1(t,x)    & := \sup_{\control{u} \in \controlmaps{U}} \EE\bigl[\ind{A}( \traj{X}{t,x}{\control{u}}{\widetilde{\tau}})\bigr] \qquad \text{where} \quad \widetilde{\tau} := \tau_B \mn T,\\
		\label{Vtilde 2}
	\widetilde{V}_2(t,x)	& := \sup_{\control{u} \in \controlmaps{U}} \EE\biggl[ \ind{A}(\traj{X}{t,x}{\control{u}}{T}) \mn \Inf{t}{r}{T} \ind{B^c}(\traj{X}{t,x}{\control{u}}{r})\biggr],\\
	    \label{Vtilde 3}
	\widetilde{V}_3(t,x)	& := \sup_{\control{u} \in \controlmaps{U}} \inf_{\sigma \in \setofst{t,T}} \E{\ind{A}(\traj{X}{t,x}{\control{u}}{T}) \mn \ind{B^c}(\traj{X}{t,x}{\control{u}}{\sigma})}.
	\end{align}
	\end{subequations}

We state the following proposition concerning assertions identical to those of Proposition \ref{prop:Reach-V1} for the reach-avoid problem of Definition \ref{def:RA terminal}.
	
	\begin{Pro}[Connection from $\wt \RA$ to \eqref{Vtilde}]
	\label{prop:Reach-Terminal}
		Let $A, B\subset R^n$ be disjoint, and suppose $B$ is closed. Then, the equality $\wt V_1 = \wt V_2 = \wt V_3$ holds on $\set S \Let  [0,T]\times \R^n$, and we have
		\begin{align*}
			\widetilde{\RA}(t,p;A,B) =  \big\{ x \in \R^n ~|~ \widetilde{V}_i(t,x) > p \big\}, \quad i \in \{1,2,3\},
		\end{align*}
		 where the set $\widetilde{\RA}$ is the set defined in Definition \ref{def:RA terminal}.
	\end{Pro}

	\begin{proof}
	    The proof follows effectively the same arguments as in the proofs of Proposition \ref{prop:Reach-V1} in \ref{app:A}.
	\end{proof}

The stochastic control problems introduced in \eqref{V1} and \eqref{Vtilde 1} are well-known as the exit-time problem \cite[p.\ 6]{SonerBook}. Note that in light of Propositions \ref{prop:Reach-V1} and \ref{prop:Reach-Terminal}, both problems in Definitions~\ref{def:RA within} and \ref{def:RA terminal} can alternatively be characterized in the framework of exit-time problems, see \eqref{V1} and \eqref{Vtilde 1}, respectively. Motivated by this, in the next section we shall focus on this class of problems. 

\section{Alternative Characterization of the Exit-Time Problem}
\label{sec:Alternative Characterization}
This section presents an alternative characterization of the exit-time problem based on solutions of certain PDEs. Let us highlight that the exit-time formulations \eqref{V1} and \eqref{Vtilde 1} involve \emph{discontinuous} payoff functions, to which the classical approaches, for example \cite{SonerBook,Krylov_ControlledDiffusionProcesses}, are not directly applicable. Consider the function
	\begin{align}
		\label{V}	V(t,x)	& := \sup_{\control{u} \in \controlmaps{U}_t} \EE\bigl[\el{ \traj{X}{t,x}{\control{u}}{\wh{\tau}(t,x)} } \bigr], \qquad\wh{\tau}(t,x) := \tau_{O}(t,x) \mn T,
	\end{align}
where the payoff function $\ell: \R^n \ra \R$ is bounded (not necessarily continuous), and $O$ is a given subset of $\R^n$. Recall that $\tau_{O}$ is the stopping time defined in Definition \ref{def:entry time} that in case of value function \eqref{V1} can be considered as $O = A \cup B$. It is immediate to observe that the functions \eqref{V1} and \eqref{Vtilde 1} are particular cases of \eqref{V} where the payoff function is $\ell(\cdot) \Let \ind{A}(\cdot)$. 
	
Hereafter we shall restrict our control processes to $\controlmaps{U}_t$, the collection of all $\Filtration_t$-progressively measurable processes $\control{u} \in \controlmaps{U}$. In view of independence of the increments of Brownian motion, the restriction of control processes to $\controlmaps{U}_t$ is not restrictive, and one can show that the function \eqref{V} remains the same if $\controlmaps{U}_t$ is replaced by $\controlmaps{U}$; see, for instance, \cite[Theorem 3.1.7, p. 132]{Krylov_ControlledDiffusionProcesses} and \cite[Remark 5.2]{BouchardTouzi_WeakDPP}.

Our objective is to characterize the function $V$ in \eqref{V} as a (discontinuous) viscosity solution of a suitable Hamilton-Jacobi-Bellman equation. 

\subsection{Assumptions and preliminaries}
For the main results of this section we need the following technical assumptions:

\begin{Assumption} We stipulate that
	\label{a:exit time}
        \begin{enumerate}[label=\alph*., align=right, widest=m, leftmargin=*]
           \item \label{a:exit time:degenerate}
            \textit{(Non-degeneracy)} The controlled processes are uniformly non-degenerate, i.e., there exists $\delta>0$ such that for all $x \in \R^n$ and $u \in \controlset{U}$, $\sigma(x,u) \sigma^{\tr}(x,u) > \delta I$ where $\sigma(x,u)$ is the diffusion term in SDE (\ref{SDE}).
           \item \label{a:exit time:Set condition}
            \textit{(Interior cone condition)} There are positive constants $h$, $r$, and an $\R^n$-value bounded map $\eta: \ol{O} \ra \R^n$ satisfying
            \begin{equation*}
                \ball{B}_{rt}\big( x+\eta(x)t\big) \subset O \qquad \text{for all $x \in \ol{O}$ and $t \in (0,h]$ }
            \end{equation*}
            where $\ball{B}_r(x)$ denotes an open ball centered at $x$ and radius $r$, and $\ol{O}$ stands for the closure of the set $O$ (see Figure \ref{Fig:Cone}).
           \item \label{a:exit time:l lsc}
            \textit{(Lower semicontinuity)} The payoff function $\ell$ in \eqref{V} is lower semicontinuous.
        \end{enumerate}
\end{Assumption}

If the set $A$ is open, then the function $\ell(\cdot) = \ind{A}(\cdot)$ as in \eqref{V1} and \eqref{Vtilde 1} satisfies Assumption \ref{a:exit time}.\ref{a:exit time:l lsc} The interior cone condition in Assumption \ref{a:exit time}.\ref{a:exit time:Set condition}\ concerns shapes of the set $ O$. Figure \ref{Fig:Cone} illustrates two typical scenarios.
	\begin{figure}
		\centering
		\subfigure[Interior cone condition holds at every point of the boundary.]{\label{Fig:IntCone}\includegraphics[scale = 1]{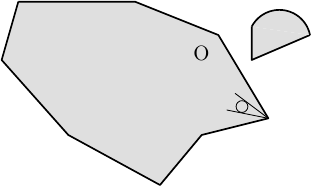}} \qquad
		\subfigure[Interior cone condition fails at the point $p$---the only possible interior cone at $p$ is a line.]{\label{Fig:NoIntCone}\includegraphics[scale = 1]{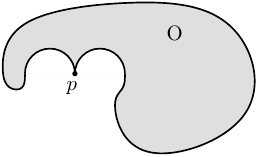}}
		\caption{Interior cone condition of the boundary.}
		\label{Fig:Cone}
	\end{figure}
	Let us define the function $J : \set{S} \times \controlmaps{U} \ra \R$:
	\begin{equation}
	    \label{J}
	    J\big(t,x,\control{u}\big) := \E{ \el{\traj{X}{t,x}{\control{u}}{\wh{\tau}(t,x)}}}, \qquad\wh{\tau}(t,x) := \tau_{O}(t,x) \mn T.
	\end{equation}
	Note that the information of the set $O$ is encoded in the definition of the stopping time $\wh \tau$. Under Assumptions \ref{a:exit time}, we establish continuity of $\wh{\tau}(t,x)$ and consequently the lower semicontinuity of $J(t,x,\control{u})$ with respect to $(t,x)$, which will be the main ingredient of our results in this section. 

\begin{Pro}[Lower semicontinuity]
    \label{prop:J1 semicontinuous}
    Consider the system (\ref{SDE}) and suppose that Assumption \ref{a:exit time} holds. Then, for any control $\control{u} \in \controlmaps{U}$ and initial condition $(t_0,x_0) \in \set{S}$, the function $(t,x) \mapsto \wh{\tau}(t,x)$ is continuous at $(t_0,x_0)$ with probability 1.\footnote{Recall that the stopping time $\wh \tau$ depends on the set $O$ which is assumed to meet the interior cone condition in Assumption \ref{a:exit time}.\ref{a:exit time:Set condition}} Moreover, the function $(t,x) \mapsto J\big(t,x,\control{u}\big)$ defined in \eqref{J} is uniformly bounded and lower semicontinuous, i.e., 
\begin{equation*}
    J\big(t,x,\control{u}\big) \leq \liminf_{(t',x') \ra (t,x)} J\big(t',x',\control{u}\big).
\end{equation*}
\end{Pro}

\begin{proof}[Sketch of the proof]
The proof essentially relies on two facts: (i) Without loss of generality, we can work with the version of the solution process which is almost sure continuous in the initial condition thanks to Kolmogorov's continuity criterion \cite[Cor.\ 1 Chap. IV, p. 220]{ref:Protter-2005} and classical inequalities concerning diffusion processes governed by SDEs \cite[Chap.\ 2]{Krylov_ControlledDiffusionProcesses}; (ii) The set of sample paths of a non-degenerate process which hits the boundary of a set satisfying Assumption \ref{a:exit time}.\ref{a:exit time:Set condition}\ and do not enter the set is negligible \cite[Corollary 3.2, p.\ 65]{ref:Bass-1998}. See \ref{app:B} for the detailed analysis.
\end{proof}
	
The main objective of this section is to provide a dynamic programming characterization of the function $V$ in \eqref{V}. To this end, given a stopping time $\theta \in \setofst{t,T}$, we need to split an admissible control onto two random intervals $[t, \theta]$ and $]\theta, T]$. The following definition formalize this separation task. Note that the control process $\control{u} \Let (u_s)_{s\ge 0} \in \controlmaps{U}_t$ at time $s \ge 0$ can be viewed as a measurable mapping $(W_{r \mx t} - W_t)_{[0,s]} \mapsto u_s \in \set{U}$, where $(W_{s})_{s\ge 0}$ is the $d$-dimensional Brownian motion in \eqref{SDE}; see \cite[Def.\ 1.11, p.\ 4]{ref:KarShr-91} for the details. Then, for $\theta \in \setofst{t,T}$ and $u \in \controlmaps{U}_t$, pathwise for any realization $\omega \in \Omega$ we define the \emph{random} policy $\control{u_\theta} \in \controlmaps{U}_{\theta(\omega)}$ as
	\begin{align}
	\label{u_theta}
		\big(W_{\cdot \mx \theta(\omega)} - W_{\theta(\omega)}\big) \mapsto \control{u}\big(W_{\cdot \mn \theta(\omega)} + W_{\cdot \mx \theta(\omega)} - W_{\theta(\omega)}\big) =: \control{u_\theta}. 
	\end{align}
Notice that $W_{.} \equiv W_{. \mn \theta(\omega)} + W_{. \mx \theta(\omega)} - W_{\theta(\omega)}$, and as such the randomness of $\control{u_\theta}$ is referred to the term $W_{.\mn \theta(\omega)}$. In view of definition \eqref{u_theta}, any admissible control $\control{u}$ can be described by
\begin{align}
\label{u_theta2}
	\control{u} = \ind{[t,\theta]}\control{u} + \ind{]\theta,T]}\control{u_{\theta}}.
\end{align}
Let us recall that by $\ind{[t,\theta]}\control{u}$, we mean that for any realization $\omega \in \Omega$ and any time $s$, we have $\ind{[t,\theta(\omega)]}{u}_s(\omega) = {u}_s(\omega)$ if $s \in [t,\theta(\omega)]$; and $= 0$ otherwise. The notation for $\ind{]\theta,T]}\control{u_\theta}$ is understood in similar fashion. It is worth noting that the relation \eqref{u_theta2} effectively implies that the random control $\control{u_\theta}$ indeed takes the same values as the control $\control{u}$ over the random time interval $]\theta,T]$. 

\begin{Lem} [Strong Markov property]
    \label{lem:Strong Markov}
      Consider the system (\ref{SDE}) whose solution process starting from $(t,x)$ controlled with $\control{u} \in \controlmaps{U}_t$ is denoted by $\traj{X}{t,x}{\control{u}}{\cdot}$. For any stopping time $\theta \in \setofst{t,T}$,  with probability one we have
    \begin{equation*}
        \E{ \el{\traj{X}{t,x}{\control{u}}{\wh{\tau}(t,x)} } ~\Big |~\sigalg_{\theta} } = \ind{ \{\wh{\tau}(t,x) < \theta \}} \el{ \traj{X}{t,x}{\control{u}}{\wh{\tau}(t,x)}}
        + \ind{\{\wh{\tau}(t,x)\geq\theta\}} J\big(\theta,\traj{X}{t,x}{\control{u}}{\theta},\control{u_\theta}\big) 
    \end{equation*}
    where $\control{u_\theta}$ is the random policy in the sense of \eqref{u_theta}, and the function $J$ and stopping time $\wh{\tau}(t,x)$ are as defined in \eqref{J}. 
\end{Lem}
\begin{proof}
    By Definition \ref{def:entry time}, we have with probability 1 that
    \begin{align*}
    \ind{\{\wh{\tau}(t,x) \ge \theta \}}\wh{\tau}(t,x) = \ind{\{ \wh{\tau}(t,x) \ge \theta \}} \big(\wh{\tau}(\theta,\traj{X}{t,x}{\control{u}}{\theta})+\theta-t\big). 
    \end{align*}
    One can now follow effectively the same computations as in the proof of \cite[Proposition 5.1]{BouchardTouzi_WeakDPP} to conclude the assertion.
\end{proof}


\subsection{Dynamic Programming Principle}\label{Section DPP 1}

The following Theorem provides a dynamic programming principle (DPP) for the exit time problem introduced in \eqref{V}.
\begin{Thm}[Dynamic Programming Principle]
\label{Theorem DPP 1}
    Consider the system \eqref{SDE} and suppose that Assumption \ref{a:exit time} holds. For any  $(t,x) \in \set{S}$ and family of stopping times $\{ \theta^{\control{u}}, \control{u} \in \controlmaps{U}_t\} \subset \setofst{t,T}$, we have
    \begin{subequations}
    \label{DPP}
    \begin{align}
    \begin{array}{r}
        \label{DPP 1 sub}  V(t,x) \leq \sup_{\control{u} \in \controlmaps{U}_t} \EE \Big[\ind{ \{\wh{\tau}(t,x) \le \theta^{\control{u}} \}} \el{ \traj{X}{t,x}{\control{u}}{\wh{\tau}(t,x)}} + \ind{\{ \wh{\tau}(t,x) > \theta^{\control{u}} \}} V^{*}\big(\theta^{\control{u}} , \traj{X}{t,x}{\control{u}}{\theta^{\control{u}}}\big)\Big],
    \end{array} 
    \end{align}
    and
    \begin{align}
    \begin{array}{r}
        \label{DPP 1 sup}  V(t,x) \geq \sup_{\control{u} \in \controlmaps{U}_t} \EE \Big[\ind{ \{\wh{\tau}(t,x) \le \theta^{\control{u}} \}} \el{ \traj{X}{t,x}{\control{u}}{\wh{\tau}(t,x)}} + \ind{\{ \wh{\tau}(t,x) > \theta^{\control{u}} \}} V_*\big(\theta^{\control{u}} , \traj{X}{t,x}{\control{u}}{\theta^{\control{u}}}\big)\Big],
    \end{array} 
    \end{align}
    \end{subequations}
    where $V$ is the function defined in \eqref{V}.
\end{Thm}

\begin{proof}
The proof is inspired by the techniques developed in \cite{BouchardTouzi_WeakDPP}, however, in the context of exit-time problems where the continuity of the exit-time (Proposition \ref{prop:J1 semicontinuous}) plays a crucial role. 

We first assemble an appropriate covering for the set $\set{S}$, and use this covering to construct an admissible control which satisfies the required conditions within $\eps$ precision, $\eps > 0$ being pre-assigned and arbitrary. For notational simplicity, in the following we set $\theta \Let \theta^{\control{u}}$.

\paragraph{Proof of \eqref{DPP 1 sub}}
	In view of Lemma \ref{lem:Strong Markov} and the tower property of conditional expectation \cite[Theorem 5.1]{ref:Kallenberg-97}, for any $(t,x) \in \set{S}$ we have
\begin{align*}
    \E{ \el{ \traj{X}{t,x}{\control{u}}{\wh{\tau}(t,x)} } }& = \E{\E{ \el{ \traj{X}{t,x}{\control{u}}{\wh{\tau}(t,x)} } \big |~\sigalg_\theta}~ } \\
    & = \E{ \ind{ \{\wh{\tau}(t,x) \le \theta \}} \el{ \traj{X}{t,x}{\control{u}}{\wh{\tau}(t,x)} }+ \ind{\{ \wh{\tau}(t,x) > \theta \}} J\big(\theta,\traj{X}{t,x}{\control{u}}{\theta},\control{u_\theta} \big)} \\
    & \leq \E{ \ind{ \{\wh{\tau}(t,x) \le \theta \}} \el{ \traj{X}{t,x}{\control{u}}{\wh{\tau}(t,x)} } + \ind{\{ \wh{\tau}(t,x) > \theta \}} V^*\big(\theta,\traj{X}{t,x}{\control{u}}{\theta}\big)},
\end{align*}
	where $\control{u_\theta}$ is the random control as introduced in \eqref{u_theta}. Note that the last inequality follows from the fact that $\control{u_\theta} \in \controlmaps{U}_{\theta(\omega)}$ for each $\omega \in \Omega$. Now taking supremum over all admissible controls $\control{u} \in \controlmaps{U}_t$ leads to the desired dynamic programming inequality \eqref{DPP 1 sub}.
	
\paragraph{Proof of \eqref{DPP 1 sup}}
	Suppose $\phi : \set{S} \ra \R$ is uniformly bounded such that
\begin{equation}
    \label{phi USC}
    \phi \in \text{USC}(\set{S}) \quad \text{and} \quad \phi \leq V_* \qquad \text{on} \quad \set{S}.
\end{equation}
According to (\ref{phi USC}) and Proposition \ref{prop:J1 semicontinuous}, given $\eps>0$, for all $(t_0,x_0) \in \set{S}$ and $\control{u} \in \controlmaps{U}_{t_0}$ there exists $r_\eps>0$ such that
\begin{equation}
    \label{C1}
    \begin{array}{cc}
    \phi(t,x)-\eps \leq \phi(t_0,x_0) \leq  V_*(t_0,x_0), & \forall (t,x) \in \ball{C}_{r_\eps}(t_0,x_0) \cap \set{S},\\
    J\big(t_0,x_0,\control{u}\big) \leq J\big(t,x,\control{u}\big) + \eps, & \forall (t,x) \in \ball{C}_{r_\eps}(t_0,x_0) \cap \set{S},
    \end{array}
\end{equation}
where $\ball{C}_{r}(t,x)$ is a cylinder defined as:
\begin{equation}
    \label{cylinder def}
    \ball{C}_r(t,x) := \{ (s,y) \in \R\times \R^n ~|~ s \in ]t-r,t] ~,~ \|x-y\| < r \}.
\end{equation}
Moreover, by definition of \eqref{J} and \eqref{V}, given $\eps>0$ and $(t_0,x_0) \in \set{S}$ there exists $\control{u}^{t_0,x_0}_\eps \in \controlmaps{U}_{t_0}     $ such that
\begin{equation*}
    V_*(t_0,x_0) \leq V(t_0,x_0) \leq J\big(t_0,x_0,\control{u}^{t_0,x_0}_\eps\big) + \eps.
\end{equation*}
By the above inequality and \eqref{C1}, one can conclude that given $\eps>0$, for all $(t_0,x_0) \in \set{S}$ there exist $\control{u}^{t_0,x_0}_\eps \in \controlmaps{U}_{t_0}$ and $r_\eps(t_0,x_0)>~0$ such that
\begin{equation}
    \label{C2}
     \phi(t,x) -3\eps \leq J\big(t,x,\control{u}^{t_0,x_0}_\eps \big)  \quad\forall (t,x) \in \ball{C}_{r_\eps(t_0,x_0)}(t_0,x_0) \cap \set{S}.
\end{equation}
Therefore, given $\eps>0$, the family of cylinders $\big\{\ball{C}_{r_{\eps}(t,x)}(t,x) ~ : ~ (t,x) \in \set{S}, \quad r_\eps(t,x) > 0 \big\}$ forms an open covering of $[0,T[ \times \R^n$. By the \textit{Lindel\"of} covering Theorem \cite[Theorem 6.3 Chapter VIII]{Topology_Dugundji}, there exists a countable sequence $(t_i,x_i,r_i)_{i \in \N}$ of elements of $\set{S} \times \R^+$ such that
\begin{equation*}
    [0,T[\times\R^n \subset \bigcup_{i\in \N}\ball{C}_{r_i}(t_i,x_i).
\end{equation*}
Note that the implication of \eqref{DPP 1 sub} simply holds for $(t,x) \in \{T\} \times \R^n$.
Let us construct a sequence $(\ball{C}^i)_{i \in \N_0}$ as
\begin{align*}
     \ball{C}^0 := \{T\} \times \R^n, \qquad
     \ball{C}^{i} := \ball{C}_{r_{i}}(t_{i},x_{i}) \setminus \! \bigcup_{j \leq i-1} \ball{C}^j.
\end{align*}
By definition $\ball{C}^i$ are pairwise disjoint and $\set{S} \subset  \bigcup_{i \in \N_0} \ball{C}^i$. Furthermore, $(\theta,\traj{X}{t,x}{\control{u}}{\theta}) \in \bigcup_{i \in \N_0} \ball{C}^i$, and for all $i \in \N_0$ there exists $\control{u}^{t_i,x_i}_\eps \in \controlmaps{U}_{t_i}$ such that
\begin{equation}
     \label{C3}
     \phi(t,x) -3\eps \leq J\big(t,x,\control{u}^{t_i,x_i}_\eps\big),  \qquad  \forall (t,x) \in \ball{C}^i \cap \set{S}.
\end{equation}
To prove \eqref{DPP 1 sup}, let us fix $\control{u}\in \controlmaps{U}_t$ and $\theta \in \setofst{t,T}$. Given $\eps >0$ we define
\begin{equation}
   \label{u concatenated}
   \control{v}_{\eps} : = \ind{[t,\theta]}\control{u} +  \ind{]\theta,T]} \sum_{i \in \N_0} \ind{\ball{C}^i} (\theta,\traj{X}{t,x}{\control{u}}{\theta}) \control{u}^{t_i,x_i}_\eps.
\end{equation}
Notice that the set of admissible controls $\controlmaps{U}_t$ (i.e., the set of $\Filtration_t$-progressively measurable functions) is closed under countable concatenation operations, and consequently $\control{v}_{\eps} \in \controlmaps{U}_t$. In light of the alternative description \eqref{u_theta2} for the control \eqref{u concatenated}, one can apply Lemma \ref{lem:Strong Markov} in conjunction with \eqref{C3} and infer that with probability 1 we have 
\begin{align*}
    \EE \Big[ \ell\big( \traj{X}{t,x}{\control{v}_\eps}{\wh{\tau}(t,x)}\big) ~\big | ~\sigalg_{\theta} \Big] & = \ind{ \{\wh{\tau}(t,x) \le \theta \}} \el{ \traj{X}{t,x}{\control{u}}{\wh{\tau}(t,x)} } +
    \ind{\{ \wh{\tau}(t,x) > \theta \}} J\Big(\theta , \traj{X}{t,x}{\control{u}}{\theta}, \sum_{i \in \N_0} \ind{\ball{C}^i} (\theta,\traj{X}{t,x}{\control{u}}{\theta})\control{u}^{t_i,x_i}_\eps\Big)\\
    & = \ind{ \{\wh{\tau}(t,x) \le \theta \}} \el{ \traj{X}{t,x}{\control{u}}{\wh{\tau}(t,x)} } + \ind{\{ \wh{\tau}(t,x) > \theta \}} \sum_{i \in \N_0} J\big(\theta , \traj{X}{t,x}{\control{u}}{\theta}, \control{u}^{t_i,x_i}_\eps\big) \ind{\ball{C}^i} \big(\theta , \traj{X}{t,x}{\control{u}}{\theta} \big)\\
    & \geq \ind{ \{\wh{\tau}(t,x) \le \theta \}} \el{ \traj{X}{t,x}{\control{u}}{\wh{\tau}(t,x)}}+ \ind{\{ \wh{\tau}(t,x) > \theta\}} \sum_{i \in \N_0} \Big(\phi\big(\theta, \traj{X}{t,x}{\control{u}}{\theta}\big)-3\eps \Big) \ind{\ball{C}^i} \big(\theta , \traj{X}{t,x}{\control{u}}{\theta} \big)\\
    & = \ind{ \{\wh{\tau}(t,x) \le \theta\}} \el{ \traj{X}{t,x}{\control{u}}{\wh{\tau}(t,x)} } + \ind{\{ \wh{\tau}(t,x) > \theta\}}
    \Big( \phi\big(\theta , \traj{X}{t,x}{\control{u}}{\theta}\big) -3\eps \Big).
\end{align*}
By the definition of $V$ and the tower property of conditional expectations,
\begin{align*}
    V(t,x) & \geq J(t,x,\control{v}_\eps) = \E{\E{\el{ \traj{X}{t,x}{\control{v}_\eps}{\wh{\tau}(t,x)} } ~\big |~\sigalg_{\theta} }}\\
    & \geq \E{\ind{ \{\wh{\tau}(t,x) \le \theta \}} \el{ \traj{X}{t,x}{\control{u}}{\wh{\tau}(t,x)}} + \ind{\{ \wh{\tau}(t,x) > \theta \}} \phi\big(\theta , \traj{X}{t,x}{\control{u}}{\theta}\big) } -3\eps ~ \E{\ind{\{\wh{\tau}(t,x)> \theta\}}}.
\end{align*}
The arbitrariness of $\control{u} \in \controlmaps{U}_t$ and $\eps >0$ implies that
\begin{align*}
    V(t,x) \geq \sup_{\control{u} \in \controlmaps{U}_t} \E{\ind{ \{\wh{\tau}(t,x) \le \theta \}} \el{ \traj{X}{t,x}{\control{u}}{\wh{\tau}(t,x)}} + \phi\big(\theta , \traj{X}{t,x}{\control{u}}{\theta}\big)}.
\end{align*}
	It suffices to find a sequence of continuous functions $(\phi_i)_{i\in \N}$ such that $\Phi_i \leq V_*$ on $\set{S}$ and converges pointwise to $V_*$. The existence of such a sequence is guaranteed by \cite[Lemma 3.5 ]{Reny_PointwiseConvergence}. 
	Thus, by Fatou's lemma,
\begin{align*}
    V(t,x) & \geq \liminf_{i \ra \infty}\sup_{\control{u} \in \controlmaps{U}_t} \E{\ind{ \{\wh{\tau}(t,x) < \theta \}} \el{ \traj{X}{t,x}{\control{u}}{\wh{\tau}(t,x)} } + \ind{\{ \wh{\tau}(t,x) \geq \theta \}} \phi_i\big(\theta , \traj{X}{t,x}{\control{u}}{\theta}\big)}\\
    & \geq \sup_{\control{u} \in \controlmaps{U}_t} \E{\ind{ \{\wh{\tau}(t,x) < \theta \}} \el{ \traj{X}{t,x}{\control{u}}{\wh{\tau}(t,x)}} + \ind{\{ \wh{\tau}(t,x) \geq \theta \}} \liminf_{i \ra \infty} \phi_i\big(\theta , \traj{X}{t,x}{\control{u}}{\theta}\big)}\\
    & = \sup_{\control{u} \in \controlmaps{U}_t} \E{\ind{ \{\wh{\tau}(t,x) < \theta \}} \el{ \traj{X}{t,x}{\control{u}}{\wh{\tau}(t,x)} }+ \ind{\{ \wh{\tau}(t,x) \geq \theta \}} V_*\big(\theta , \traj{X}{t,x}{\control{u}}{\theta}\big)}.
\end{align*}
\end{proof}

\begin{Remark}[Measurability]
\label{rem:DPP}
The DPP in \eqref{DPP} is introduced in a weaker sense than the standard DPP for stochastic optimal control problems \cite{SonerBook}. Namely, one does not have to verify the measurability of the function $V$ in (\ref{V}) to apply \eqref{DPP}.
\end{Remark}


\subsection{Dynamic Programming Equation}\label{sec:PDE}
Our objective in this subsection is to demonstrate how the DPP derived in \secref{Section DPP 1} characterizes the function $V$ as a (discontinuous) viscosity solution to an appropriate HJB equation; for the general theory of viscosity solutions we refer to \cite{Crandali_Ishii_Lions_VoscositySolutions} and \cite{SonerBook}. To complete the PDE characterization and provide numerical solutions for this PDE, one also needs appropriate boundary conditions which will be the objective of the next subsection.

\begin{Def}[Dynkin operator]
\label{def:Dynkin operator}
Given $u\in \controlset{U}$, we denote by $\mathcal{L}^u$ the Dynkin operator (also known as the infinitesimal generator) associated to the controlled diffusion \eqref{SDE} as
\begin{align*}
 \mathcal{L}^u \Phi(t,x) := \partial_t \Phi(t,x) &+ f(x,u).\partial_x \Phi(t,x) + \frac{1}{2}\text{Tr}[\sigma \sigma^\top(x,u) \partial_{x}^2 \Phi(t,x)],
\end{align*}
where $\Phi$ is a real-valued function smooth on the interior of $\set{S}$, with $\partial_t \Phi$ and $\partial_x \Phi$ denoting the partial derivatives with respect to $t$ and $x$ respectively, and $\partial^2_x \Phi$ denoting the Hessian matrix with respect to $x$.
\end{Def}


\begin{Thm}[Dynamic Programming Equation]
\label{Theorem DPE 1}
Consider the system \eqref{SDE} and suppose that Assumption \ref{a:exit time} holds. Then,
\begin{itemize}[label=$\circ$, leftmargin=*]
	\item the lower semicontinuous envelope of $V$ introduced in \eqref{V} is a viscosity supersolution of
		\begin{equation*}
		    \label{supersolution V1}
		    -\sup_{u \in \controlset{U}} \mathcal{L}^u V_*(t,x) \geq 0  \qquad \text{on} \quad [0,T[ \times {\ol{O}}^c,
		\end{equation*}
	\item the upper semicontinuous envelope of $V$ is a viscosity subsolution of
		\begin{equation*}
		 \label{supersolution}
		 -\sup_{u \in \controlset{U}} \mathcal{L}^u V^{*}(t,x) \leq 0  \qquad \text{on} \quad [0,T[\times {\ol{O}}^c,
		\end{equation*}
\end{itemize}
\end{Thm}

\begin{proof}
We first prove the supersolution part:

\textbf{Supersolution:} For the sake of contradiction, assume that there exists $(t_0,x_0) \in [0,T[ \times \ol{O}^c$ and a smooth function $\phi: \set{S} \ra \R$ satisfying
\begin{align*}
    \min_{(t,x) \in \set{S}}\big(V_*-\phi \big)(t,x) = \big(V_*-\phi \big)(t_0,x_0) = 0
\end{align*}
such that for some $\delta>0$
\begin{equation*}
    -\sup_{u \in \controlset{U}} \mathcal{L}^u \phi(t_0,x_0) < -2\delta
\end{equation*}
Notice that, without loss of generality, one can assume that $(t_0,x_0)$ is the strict minimizer of $V_*-\phi$ \cite[Lemma II 6.1, p.\ 87]{SonerBook}. Since $\phi$ is smooth, the map $(t,x) \mapsto \mathcal{L}^u \phi (t,x)$ is continuous. Therefore, there exist $u \in \controlset{U}$ and $r>0$ such that $\ball{B}_r(t_0,x_0) \subset [0,T) \times \ol{O}^c$ and
\begin{equation}
    \label{C4}
    \begin{array}{cc}
    -\mathcal{L}^u \phi(t,x) < -\delta & \forall (t,x) \in \ball{B}_r(t_0,x_0).
    \end{array}
\end{equation}
Let us define the stopping time $\theta(t,x)\in \setofst{t,T}$
\begin{equation}
    \label{theta}
    \theta(t,x) = \inf\{ s\geq t ~:~ (s,\traj{X}{t,x}{u}{s}) \notin \ball{B}_r(t_0,x_0) \},
\end{equation}
where $(t,x) \in \ball{B}_r(t_0,x_0)$.  Note that by continuity of solutions to \eqref{SDE}, $t < \theta(t,x) < T$ $\PP$- a.s.\ for all $(t,x) \in \ball{B}_r(t_0,x_0)$. Moreover, selecting $r>0$ sufficiently small so that $\theta(t,x)<\tau_{O}$, we have
\begin{equation}
    \label{theta property}
    \begin{array}{ccc}
    \theta(t,x) < \tau_{O} \mn T = \wh{\tau}(t,x)  &  \PP\text{-a.s.} & \forall (t,x) \in \ball{B}_r(t_0,x_0)
    \end{array}
\end{equation}
Applying It\^{o}'s formula and using \eqref{C4}, we see that for all $(t,x) \in \ball{B}_r(t_0,x_0)$,
\begin{align*}
    \phi(t,x) & = \E{\phi\big(\theta(t,x),\traj{X}{t,x}{u}{\theta(t,x)}\big) + \int^{\theta(t,x)}_{t}-\mathcal{L}^u \phi\big(s,\traj{X}{t,x}{u}{s}\big) ds } \\
                            & \leq \E{\phi\big(\theta(t,x),\traj{X}{t,x}{u}{\theta(t,x)}\big)} - \delta (\E{\theta(t,x)}-t) \\
                            & < \E{\phi\big(\theta(t,x),\traj{X}{t,x}{u}{\theta(t,x)}\big)}.
\end{align*}
Now it suffices to take a sequence $(t_n,x_n,V(t_n,x_n))_{n \in \N}$ converging to $(t_0,x_0,V_*(t_0,x_0))$ to see that
\begin{equation*}
\phi(t_n,x_n) \ra \phi(t_0,x_0) = V_*(t_0,x_0).
\end{equation*}
Therefore, for sufficiently large $n$ we have
\begin{align*}
    V(t_n,x_n) < \E{\phi\big(\theta(t_n,x_n),\traj{X}{t_n,x_n}{u}{\theta(t_n,x_n)}\big)} < \E{V_*\big(\theta(t_n,x_n),\traj{X}{t_n,x_n}{u}{\theta(t_n,x_n)}\big)},
\end{align*}
which, in accordance with \eqref{theta property}, can be expressed as
\begin{align*}
    V(t_n,x_n) < \EE\Big[& \ind{\{\wh{\tau}(t_n,x_n) < \theta(t_n,x_n) \}} \el{ \traj{X}{t_n,x_n}{u}{\wh{\tau}(t_n,x_n)} } + \ind{\{ \wh{\tau}(t_n,x_n) \geq \theta(t_n,x_n) \}} V_*\big(\theta , \traj{X}{t_n,x_n}{u}{\theta(t_n,x_n)}\big)\Big].
\end{align*}
This contradicts the DPP in \eqref{DPP 1 sup}.

\textbf{Subsolution:} The subsolution property is proved in a fashion similar to the supersolution part but with slightly more care. For the sake of contradiction, assume that there exists $(t_0,x_0) \in [0,T[ \times \ol{O}^c$ and a smooth function $\phi: \set{S} \ra \R$ satisfying
\begin{align*}
    \max_{(t,x) \in \set{S}}\big(V^*-\phi \big)(t,x) = \big(V^*-\phi \big)(t_0,x_0) = 0
\end{align*}
such that for some $\delta>0$
\begin{equation*}
    -\sup_{u \in \controlset{U}} \mathcal{L}^u \phi(t_0,x_0) > 2\delta.
\end{equation*}
By continuity of the mapping $(t,x,u) \mapsto \mathcal{L}^u \phi (t,x)$ and compactness of the control set $\set{U}$, there exists $r>0$ such that for all $u \in \controlset{U}$
\begin{equation}
    \label{C5}
    \begin{array}{cc}
    -\mathcal{L}^u \phi(t,x) > \delta, & \forall (t,x) \in \ball{B}_r(t_0,x_0),
    \end{array}
\end{equation}
where  $\ball{B}_r(t_0,x_0) \subset [0,T) \times \ol{O}^c$. Note as in the preceding part, $(t_0,x_0)$ can be considered as the strict maximizer of $V^* - \phi$ that consequently implies that there exists $\gamma > 0$ such that 
\begin{equation}\label{C6}	
	\big(V^* - \phi \big)(t,x) < -\gamma, \qquad \forall (t,x) \in \partial \ball{B}_r(t_0,x_0).
\end{equation}
where $\partial \ball{B}_r(t_0,x_0)$ stands for the boundary of the ball $\ball{B}_r(t_0,x_0)$. Let $\theta(t,x) \in \setofst{t,T}$ be the stopping time defined in \eqref{theta}; notice that $\theta$ may, of course, depend on the policy $\control{u}$. Applying It\^{o}'s formula and using \eqref{C5}, one can observe that given $\control{u} \in \controlmaps{U}_t$,
\begin{align*}
    \phi(t,x) & = \E{\phi\big(\theta(t,x),\traj{X}{t,x}{\control{u}}{\theta(t,x)}\big) + \int^{\theta(t,x)}_{t}-\mathcal{L}^{u_s} \phi\big(s,\traj{X}{t,x}{\control{u}}{s}\big) ds } \\
                            & \geq \E{\phi\big(\theta(t,x),\traj{X}{t,x}{\control{u}}{\theta(t,x)}\big)} + \delta (\E{\theta(t,x)}-t) \\
                            & > \E{\phi\big(\theta(t,x),\traj{X}{t,x}{\control{u}}{\theta(t,x)}\big)}.
\end{align*}
Now it suffices to take a sequence $(t_n,x_n,V(t_n,x_n))_{n \in \N}$ converging to $(t_0,x_0,V^*(t_0,x_0))$ to see that
\begin{equation*}
\phi(t_n,x_n) \ra \phi(t_0,x_0) = V^*(t_0,x_0).
\end{equation*}
As argued in the supersolution part above, for sufficiently large $n$, for given $\control{u} \in \controlmaps{U}_t$,
\begin{align*}
    V(t_n,x_n) > \E{\phi\big(\theta(t_n,x_n),\traj{X}{t_n,x_n}{\control{u}}{\theta(t_n,x_n)}\big)} > \E{V^*\big(\theta(t_n,x_n),\traj{X}{t_n,x_n}{\control{u}}{\theta(t_n,x_n)}\big)} +\gamma,
\end{align*}
where the last inequality is deduced from the fact that $\big(\theta(t_n,x_n),\traj{X}{t_n,x_n}{\control{u}}{\theta(t_n,x_n)}\big) \in \partial \ball{B}_r(t_0,x_0)$ together with \eqref{C6}. Thus, in view of \eqref{theta property}, we arrive at 
\begin{align*}
    V(t_n,x_n) > \EE\Big[& \ind{ \{\wh{\tau}(t,x) < \theta(t_n,x_n) \}} \el{ \traj{X}{t_n,x_n}{\control{u}}{\wh{\tau}}} + \ind{\{ \wh{\tau}(t,x) \geq \theta(t_n,x_n) \}} V^*\big(\theta , \traj{X}{t_n,x_n}{\control{u}}{\theta(t_n,x_n)}\big)\Big] + \gamma.
\end{align*}
This contradicts the DPP in \eqref{DPP 1 sub} as $\gamma$ is chosen uniformly with respect to $\control{u} \in \controlmaps{U}_t$.
\end{proof}

\subsection{Boundary conditions}
Before proceeding with the main result of this subsection on boundary conditions, we need a preparatory result that indeed has a stronger assertion than Proposition \ref{prop:J1 semicontinuous}.

\begin{Pro}[Uniform continuity]
\label{prop:unif cont}
	Under the same hypothesis of Proposition \ref{prop:J1 semicontinuous}, for any sequence of control policies $(\control{u_n})_{n\in \N} \subset \controlmaps{U}_t$ and initial conditions $(t_n,x_n) \ra (t,x)$, we have
		\begin{align*}
			\lim_{n \ra \infty} \Big \| \traj{X}{t,x}{\control{u}_n}{\wh{\tau}(t,x)} - \traj{X}{t_n,x_n}{\control{u}_n}{\wh \tau (t_n,x_n)} \Big \| = 0, \qquad \PP \text{-a.s.}, 
		\end{align*}
		where the stopping time $\wh \tau$ is introduced in \eqref{V}.
\end{Pro}

\begin{proof}
The proof follows the same lines as in the proof of Proposition \ref{prop:J1 semicontinuous}, but in a uniform fashion with respect to admissible control inputs; see \ref{app:B} for the details. 
\end{proof}

The following theorem provides boundary conditions for the function $V$ both in viscosity and Dirichlet (pointwise) senses:

\begin{Thm}[Boundary conditions]
	\label{thm:boundary}
	Suppose that the condition of Theorem \ref{Theorem DPE 1} holds. Then the function $V$ in \eqref{V} satisfies the following boundary value conditions:
		\begin{subequations}
		\label{boundary}		
			\begin{align}
			\label{boundary pointwise} 
			\text{Dirichlet:}&  \quad 
			\begin{cases} 
			{V}(t,x) = \ell(x) \\
			\forall (t,x) \in  [0,T] \times  \ol O \bigcup \{T\} \times \R^n \\
			\end{cases} \\
			\label{boundary visc} 
			\text{Viscosity:}&  \quad 
			\begin{cases}
			\limsup\limits_{\footnotesize \begin{smallmatrix} (\ol{O})^c \ni x' \ra x \\ t' \ua t \end{smallmatrix}} {V}(t',x') \le  \ell^*(x) \\
				\liminf\limits_{\footnotesize \begin{smallmatrix}  (\ol{O})^c  \ni x' \ra x \\ t' \ua t
				\end{smallmatrix}}{V}(t',x') \ge  \ell(x) \\
				\forall (t,x) \in [0,T] \times \partial O \bigcup \{T\} \times \R^n 
			\end{cases}
			\end{align}
		\end{subequations}
\end{Thm}

\begin{proof}
In light of \cite[Corollary 3.2, p.\ 65]{ref:Bass-1998}, Assumptions \ref{a:exit time}.\ref{a:exit time:degenerate} and \ref{a:exit time}.\ref{a:exit time:Set condition} ensure that 
	\begin{equation*}
	    \wh{\tau}(t,x) = t, \quad \forall (t,x) \in [0,T] \times \ol{O} \cup \{T\} \times \R^n  \qquad \PP \text{-a.s.}
	\end{equation*}
which readily implies the pointwise boundary condition \eqref{boundary pointwise}. To prove the discontinuous viscosity boundary condition \eqref{boundary visc}, we only show the first assertion; the second one follows from similar arguments. Let $(t,x) \in [0,T] \times  \partial O \bigcup \{T\} \times \R^n$ and $(t_n,x_n) \ra (t,x)$, where $t_n < T$ and $x \in (\ol{O})^c$. In the definition of $V$ in \eqref{V}, one can choose a sequence of policies that is increasing and attains the supremum value. This sequence, of course, depends on the initial condition. Thus, let us denote it via two indices $(\control{u}_{n,j})_{j \in \N}$ as a sequence of policies corresponding to the initial condition $(t_n,x_n)$ corresponding to the value $V(t_n,x_n)$. In this light, there exists a subsequence of $(\control{u}_{n_j})_{j \in \N}$ such that
	\begin{subequations}
		\begin{align}
			V^*(t,x) & = \lim_{n\ra \infty} V(t_n,x_n)  = \lim_{n \ra \infty}\lim_{j \ra \infty} \EE \Big[ \ell \big(\traj{X}{t_n,x_n}{\control{u}_{n,j}}{\wh{\tau}(t_n,x_n)}\big) \Big] \nonumber\\
			& \label{lem 1} \le \lim_{j \ra \infty} \EE \Big[ \ell \big( \traj{X}{t_{j},x_{j}}{\control{u}_{n_j}}{\wh{\tau}(t_{j},x_{j})} \big) \Big] 
			 \le  \EE \Big[\lim_{j \ra \infty} \ell \big( \traj{X}{t_{j},x_{j}}{\control{u}_{n_j}}{\wh{\tau}(t_{j},x_{j})} \big) \Big] \\
			 & \label{lem 2} \le \ell^*(x)
		\end{align}
	\end{subequations}
where the second inequality in \eqref{lem 1} follow from Fatou's lemma, and \eqref{lem 2} if the consequence of the almost sure uniform continuity assertion in Proposition \ref{prop:unif cont}. Let us recall that $\wh \tau(t,x) = t$ and consequently $\traj{X}{t,x}{\control{u}_{n_j}}{\wh \tau(t,x)} = x$.
\end{proof}

Theorem \ref{thm:boundary} provides boundary condition for $V$ in both Dirichlet (pointwise) and viscosity senses. The Dirichlet boundary condition \eqref{boundary pointwise} is the one usually employed to numerically compute the solution via PDE solvers, whereas the viscosity boundary condition \eqref{boundary visc} is required for theoretical support of the numerical schemes and comparison results. 

\section{Connection Between the Reach-Avoid Problem and PDE Characterization}
\label{sec:connection 3 to 4}
In this section we draw a connection between the reach-avoid problem of \secref{sec:ProblemStatement} and the stochastic optimal control problems stated in \secref{sec:connection}. This connection for the problem of reach-avoid at the terminal time $T$ (Definition \ref{def:RA terminal}) is straightforward, as it only suffices to ensure that the target set $A$ is open and the avoid set $B$ is closed. Namely, set $B$ being closed fulfills the requirement of Proposition \ref{prop:Reach-Terminal} that bridges the problem $\wt \RA$ to optimal control $\wt V_1$ in \eqref{Vtilde 1}. On the other hand, set $A$ being open guarantees that the payoff function $\ind{A}$ meets the lower semicontinuity of Assumption~\ref{a:exit time}\ref{a:exit time:l lsc}, which allows to deploy the PDE characterization developed in \secref{sec:Alternative Characterization} (i.e., Theorem \ref{Theorem DPE 1} together with boundary conditions in Theorem \ref{thm:boundary}) to approach $\wt V_1$ in \eqref{Vtilde 1} for numerical purposes.

However, the above discussion does not immediately apply to the reach-avoid problem within $[t,T]$ (Definition \ref{def:RA within}). That is, Proposition \ref{prop:Reach-V1} imposes a constraint on both sets $A$ and $B$ to be closed, which is clearly in contradiction with the lower semicontinuity of the payoff function $\ell$ in \eqref{V}. 

To achieve a reconciliation between the two sets of hypotheses in case of Definition \ref{def:RA within}, given closed sets $A$ and $B$, we construct a smaller set $A_\eps\subset A^\circ$ where $A_\eps \Let \{x\in A\mid \dist(x, A^c) \ge \eps\}$ \footnote{$\dist(x,A) \Let \inf_{y \in A}\|x - y\|$, where $\|\cdot\|$ stands for the Euclidean norm.} and $A_\eps$ satisfies Assumption~\ref{a:exit time}.\ref{a:exit time:Set condition} Note that this is always possible if $O \Let A\cup B$ satisfies Assumption \ref{a:exit time}.\ref{a:exit time:Set condition}---indeed, simply take $\eps < h/2$ to see this, where $h$ is as defined in Assumption~\ref{a:exit time}.\ref{a:exit time:Set condition} Figure \ref{Fig:EpsPrecision} depicts this case.
	\begin{figure}
		\begin{center}
			\includegraphics[scale = .9]{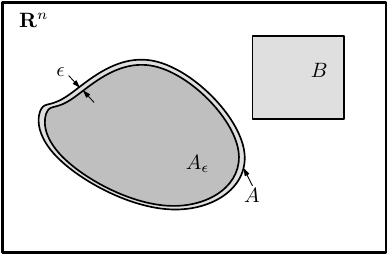}
		\end{center}
		\caption{Construction of the sets $A_\eps$ from $A$ as described in \secref{sec:connection 3 to 4}.}
		\label{Fig:EpsPrecision}
	\end{figure}
	To be precise, we define
	\begin{equation}
        \label{V eps}
		V_\eps(t, x) \Let \sup_{\control{u}\in\controlmaps{U}_t}\E{\ell_\eps \bigl(\traj{X}{t, x}{\control{u}}{\tau_\eps}\bigr)},\qquad \tau_\eps\Let \tau_{A_\eps\cup B} \mn T,
    \end{equation}
	where the function $\ell_\eps :\R^n \ra \R$ is defined as 
	\begin{align*}
		\ell_\eps(x) \Let\biggl( 1 - \frac{\dist(x,A_\eps)}{\eps}\biggr) \mx 0 .
	\end{align*}

The following result asserts that the above technique affords a conservative but arbitrarily precise way of characterizing the solution of the reach-avoid problem defined in Definition \ref{def:RA within} in the framework of \secref{sec:Alternative Characterization}.

    \begin{Thm}[Approximation stability]
	\label{thm:approx}
			Consider the system \eqref{SDE}, and suppose the sets $A, B$ are closed and Assumptions \ref{a:exit time}.\ref{a:exit time:degenerate} and \ref{a:exit time}.\ref{a:exit time:Set condition} hold. For all $(t,x) \in [t,T[ \times \R^n$ and $\eps_1 \ge  \eps_2 >0$, we have $V_{\eps_2}(t,x) \ge V_{\eps_1}(t,x)$, and $V(t,x) = \lim_{\eps \da 0} V_\eps (t,x)$ where the functions $V$ and $V_\eps$ are defined as \eqref{V1} and \eqref{V eps}, respectively.
    \end{Thm}

    \begin{proof}
        By definition, the family of the sets $(A_\eps)_{\eps>0}$ is nested and increasing as $\eps \da 0$. Therefore, in view of \eqref{rem:A cup B}, $\tau_\eps$ is nonincreasing as $\eps \da 0$ pathwise on $\Omega$. Moreover it is obvious to see that the family of functions $\ell_\eps$ is increasing with respect to $\eps$. Hence, given an initial condition $(t,x) \in \set{S}$, an admissible control $\control u \in \controlmaps{U}_t$, and $\eps_1 \ge  \eps_2 >0$, pathwise on $\Omega$ we have
        \begin{align*}
            \ell_{\eps_2}\big( \traj{X}{t,x}{\control u}{\tau_{\eps_2}} \big) < 1 &\Lra \tau_{\eps_2} = \tau_B \mn T < \tau_{A_{\eps_2}} < \tau_{A_{\eps_1}} \\
            & \Lra \tau_{\eps_1} = \tau_B \mn T = \tau_{\eps_2} \\
            & \Lra \ell_{\eps_2}\big( \traj{X}{t,x}{\control u}{\tau_{\eps_2}} \big) \ge \ell_{\eps_1}\big( \traj{X}{t,x}{\control u}{\tau_{\eps_1}} \big),
        \end{align*}
        which immediately leads to $V_{\eps_2}(t,x) \ge V_{\eps_1}(t,x)$. Now let $(\eps_i)_{i\in\N}$ be a decreasing sequence of positive numbers that converges to zero, and for the simplicity of notation let $A_n \Let A_{\eps_n}$, $\tau_n \Let \tau_{\eps_n}$, and $\ell_n \Let \ell_{\eps_n}$. According to the definitions \eqref{V1} and \eqref{V eps}, we have
        \begin{subequations}
        \begin{align}
            V(t,x) & - \lim_{n \ra \infty}  V_{\eps_n}(t,x) = \sup_{\control u \in \controlmaps{U}_t} \E{ \ind{A}\big({\traj{X}{t,x}{\control u}{\wh \tau}}\big) } - \lim_{n \ra \infty} \sup_{\control u \in \controlmaps{U}_t} \E{ \ell_n\big({\traj{X}{t,x}{\control u}{\tau_n}}\big) } \nonumber \\
            & \label{1} = \sup_{\control u \in \controlmaps{U}_t} \E{ \ind{A}\big({\traj{X}{t,x}{\control u}{\wh \tau}}\big) } - \sup_{n \in \N} \sup_{\control u \in \controlmaps{U}_t} \E{\ell_n \big({\traj{X}{t,x}{\control u}{\tau_n}}\big)} \\
            & \le \sup_{\control u \in \controlmaps{U}_t} \Big(\E{ \ind{A}\big({\traj{X}{t,x}{\control u}{\wh \tau}}\big) } - \sup_{n \in \N} \E{ \ell_n\big({\traj{X}{t,x}{\control u}{\tau_n}}\big)} \Big)\nonumber \\
            & \le \sup_{\control u \in \controlmaps{U}_t} \inf_{n \in \N} \E{ \ind{A}\big({\traj{X}{t,x}{\control u}{\wh \tau}}\big) - \ind{A_n}\big({\traj{X}{t,x}{\control u}{\tau_n}}\big)} \nonumber \\
            &\label{2} = \sup_{\control u \in \controlmaps{U}_t} \inf_{n \in \N} \PP \Big( \{\tau_{A_n} > \tau_B \mn T \} \cap \{\tau_A \le T\}  \cap\{\tau_A < \tau_B\}  \Big) \\
            &\label{3} = \sup_{\control u \in \controlmaps{U}_t} \PP \Big(\bigcap_{n \in \N}  \{\tau_{A_n} > \tau_B \mn T \} \cap \{\tau_A \le T\}  \cap\{\tau_A < \tau_B\} \Big)\\
            &\label{4} \le \sup_{\control u \in \controlmaps{U}_t} \PP \Big( \{\tau_{A^\circ} \ge \tau_B \mn T \} \cap \{\tau_A \le T\}  \cap\{\tau_A < \tau_B\}\Big) \\          
            &\label{5} \le \sup_{\control u \in \controlmaps{U}_t} \PP \big( \{\tau_{A^\circ} > \tau_A \} \cup\{\tau_A = T\} \big) = 0 
        \end{align}
        \end{subequations}
        Note that the equality in \eqref{1} is due to the fact that the sequence of the functions $\big(V_{\eps_n}\big)_{n\in \N}$ is increasing pointwise. One can infer the equality \eqref{2} when $\ind{A}\big({\traj{X}{t,x}{\control u}{\wh \tau}}\big) = 1$ and $\ind{A_n}\big({\traj{X}{t,x}{\control u}{\tau_n}}\big) = 0$ as $\ind{A}\big({\traj{X}{t,x}{\control u}{\wh \tau}}\big) \ge \ind{A_n}\big({\traj{X}{t,x}{\control u}{\tau_n}}\big)$ pathwise on $\Omega$. Moreover, since the sequence of the stopping times $(\tau_n)_{n\in\N}$ is decreasing $\PP$-a.s., the family of sets $\big(\{\tau_{A_n} > \tau_A \}\big)_{n\in\N}$ is also decreasing; consequently, the equality \eqref{3} follows. In order to show \eqref{4}, it is not hard to inspect that
        \begin{align*}
        	\omega \in \bigcap_{n \in \N}  \{\tau_{A_n} > \tau_B \mn T \} & \Lra \forall n \in \N, \quad \tau_{A_n}(\omega) > \tau_B(\omega)\mn T \\ 
			& \Lra \forall n \in \N, \quad \forall s \le \tau_B(\omega)\mn T, \quad \traj{X}{t,x}{\control{u}}{s}(\omega) \notin A_n \\
			& \Lra \forall s \le \tau_B(\omega)\mn T, \quad \traj{X}{t,x}{\control{u}}{s}(\omega) \notin \bigcup_{n\in \N} A_n = A^\circ \\
			& \Lra \omega \in \{\tau_{A^\circ} \ge \tau_B \mn T\}.      	
        \end{align*}   
	Based on non-degeneracy and the interior cone condition in Assumptions \ref{a:exit time}.\ref{a:exit time:degenerate} and \ref{a:exit time}.\ref{a:exit time:Set condition} respectively, by virtue of \cite[Corollary 3.2, p.\ 65]{ref:Bass-1998}, we see that the set $\{\tau_{A^\circ} > \tau_A\}$ is negligible. Moreover, the interior cone condition implies that the Lebesgue measure of $\partial A$, boundary of $A$, is zero. In view of non-degeneracy and Girsanov's Theorem \cite[Theorem 5.1, p.\ 191]{ref:KarShr-91}, $\traj{X}{t,x}{\control u}{r}$ has a probability density $d(r,y)$ for $r \in ]t,T]$; see \cite[Section IV.4]{SonerBook} and references therein. Hence, the aforesaid property of $\partial A$ results in $\PP\{\tau_A = T\} \leq \PP\big \{ \traj{X}{t,x}{\control{u}}{T} \in \partial A \big \} = \int_{\partial A} d(T,y)\diff y = 0$, and the second equality of \eqref{5} follows. It is straightforward to see $V \ge V_{\eps_n}$ pointwise on $\set S$ for all $n\in\N$. The assertion now follows at once.
  \end{proof}

The following corollary asserts the application of the results developed in \secref{sec:Alternative Characterization} to the function $V_\eps$ in \eqref{V eps}. The corollary not only simplifies the PDE characterization developed in \secref{sec:PDE} from discontinuous to continuous regime, but also provides a theoretical justification for deployment of existing PDE solvers (e.g., \cite{ref:Mitchell-toolbox}) for numerical purposes. This result in fact coincides with classical stochastic optimal control when the payoff function is continuous \cite[Theorem 8.2]{Crandali_Ishii_Lions_VoscositySolutions}.

\begin{Cor}[Continuous regime]
\label{cor:PDE_eps}
	Consider the system in \eqref{SDE} and suppose that Assumption \ref{a:exit time} holds. Then, for any $\eps>0$ the function $V_\eps:\set{S} \ra [0,1] $ in \eqref{V eps} is continuous. Furthermore, if  $(A_\eps \cup B)^c$ is bounded\footnote{One may replace this condition by imposing the drift and diffusion terms to be bounded.} then $V_\eps$ is the unique viscosity solution of 
	\begin{align}
		\label{PDE eps}
		 \begin{cases}
			 -\sup\limits_{u \in \controlset{U}} \mathcal{L}^u V_\eps(t,x) = 0  \qquad & \text{in} \quad [0,T[ \times (A_\eps \cup B)^c \\
		 	 V_\eps(t,x) = \ell_\eps(x) \qquad & \text{on} \quad    [0,T] \times (A_\eps \cup B)   \bigcup  \{ T \} \times \R^n 
 		 \end{cases}
	\end{align}
\end{Cor}

\begin{proof}
	The continuity of the function $V_\eps$ defined as in \eqref{V eps} readily follows from Lipschitz continuity of the payoff function $\ell_\eps$ and uniform continuity of the stopped solution process in Proposition \ref{prop:unif cont}.\footnote{This continuity result can, alternatively, be deduced via the comparison result of the viscosity characterization of Theorem \ref{Theorem DPE 1} together with boundary conditions \eqref{boundary visc} \cite{Crandali_Ishii_Lions_VoscositySolutions}.} The PDE characterization of $V_\eps$ in \eqref{PDE eps} is the straightforward consequence of its continuity and Theorem \ref{Theorem DPE 1} with boundary condition in Theorem \ref{thm:boundary}. The uniqueness follows from the weak comparison principle, \cite[Theorem VII.8.1, p.\ 274]{SonerBook}, that in fact requires $(A_\eps \cup B)^c$ being bounded.
\end{proof}

Let us remark that under further regularity conditions on the payoff function (i.e., differentiability), the assertion of Corollary~\ref{cor:PDE_eps} may be even more strengthened in which the PDE is understood in the classical sense; see for example \cite[Theorem VI.5.1, p.\ 238]{SonerBook} for further details. The following Remark summarizes the preceding results and pave the analytical ground so that the Reach-Avoid problem is amenable to numerical solutions by means of off-the-shelf PDE solvers.

\begin{Remark}[Numerical stability]
	Theorem \ref{thm:approx} implies that the conservative approximation $V_\eps$ can be arbitrarily precise, i.e., $V(t,x) = \lim_{\eps \downarrow 0} V_\eps(t,x)$. Corollary \ref{cor:PDE_eps} implies that $V_\eps$ is continuous, i.e., the PDE characterization in Theorem \ref{Theorem DPE 1} can be simplified to the continuous version. Continuous viscosity solution can be numerically solved by invoking existing toolboxes, e.g. \cite{ref:Mitchell-toolbox}. The precision of numerical solutions can also be arbitrarily accurate at the cost of computational time and storage. In other words, let $V^\delta_\eps$ be the numerical solution of $V_\eps$ obtained through a numerical routine, and let $\delta$ be the descretizaion parameter (grid size) as required by \cite{ref:Mitchell-toolbox}. Then, since the continuous PDE characterization meets the hypothesis required for the toolbox \cite{ref:Mitchell-toolbox}, we have $V_\eps = \lim_{\delta \downarrow 0} V^\delta_\eps$, and consequently we have $V(t,x) = \lim_{\eps \downarrow 0} \lim_{\delta \downarrow 0} V^\delta_\eps(t,x)$.
\end{Remark}

\section{Numerical Example: Zermelo Navigation Problem}
\label{sec:simulation}
	To illustrate the theoretical results  of the preceding sections, we apply the proposed reach-avoid formulation to  the \textit{Zermelo} navigation problem with constraints and stochastic uncertainties. In control theory, the Zermelo navigation problem consists of a swimmer who aims to reach an island (Target) in the middle of a river while avoiding the waterfall, with the river current leading towards the waterfall. The situation is depicted in Figure \ref{Fig:waterfall}.
	\begin{figure}
    	\begin{center}
		\includegraphics[scale = 0.7]{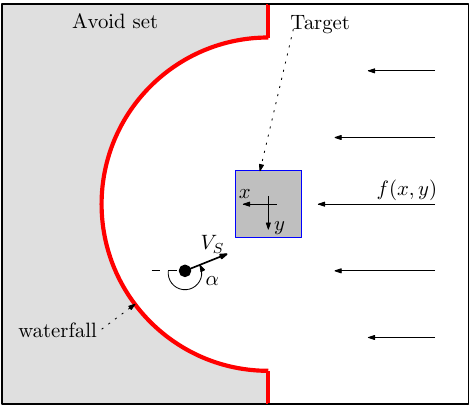}
		\caption{Zermelo navigation problem : a swimmer in the river}
		\label{Fig:waterfall}    	
    	\end{center}
	\end{figure}		
	We say that the swimmer ``succeeds'' if he reaches the target before going over the waterfall, the latter forming a part of his Avoid set.

\subsection{Mathematical modeling}
	The dynamics of the river current are nonlinear; we let $f(x,y)$ denote the river current at position $(x,y)$ \cite{ref:CardaliaguetQuincampoixSaint-1997}. We assume that the current flows with constant direction towards the waterfall, with the magnitude of $f$ decreasing in distance from the middle of the river: 
	\begin{equation*}
	    f(x,y) \Let \begin{bmatrix}
	                1 - ay^2 \\
	                0 \\
	              \end{bmatrix}.
	\end{equation*}
	To describe the uncertainty of the river current, we consider the diffusion term 
	\begin{equation*}
	    \sigma(x,y) \Let \begin{bmatrix}
	                \sigma_x & 0 \\
	                0 & \sigma_y \\
	              \end{bmatrix}.
	\end{equation*}
	We assume that the swimmer moves with constant velocity $V_S$, and we assume that he can change his direction $\alpha$ instantaneously. The complete dynamics of the swimmer in the river is given by
\begin{equation}
\label{river dynamic}
		\begin{bmatrix}
            \diff x_s \\
            \diff y_s \\
        \end{bmatrix} = \begin{bmatrix}
                            1 - ay^2 + V_S \cos(\alpha)\\
                                V_S \sin(\alpha) \\
                        \end{bmatrix} \diff s + \begin{bmatrix}
                                                  \sigma_x & 0 \\
                                                  0 & \sigma_y \\
                                                \end{bmatrix} \diff W_s,
\end{equation}
where $W_s$ is a two-dimensional Brownian motion, and $\alpha \in [\pi, \pi]$ is the direction of the swimmer with respect to the $x$ axis and plays the role of the controller for the swimmer.

\subsection{Reach-Avoid formulation}
	Obviously, the probability of the swimmer's ``success'' starting from some initial position in the navigation region depends on starting point $(x,y)$. As shown in \secref{sec:connection}, this probability can be characterized as the level set of a function, and by Theorem \ref{Theorem DPE 1} this function is the discontinuous viscosity solution of a certain differential equation on the navigation region with particular lateral and terminal boundary conditions. The differential operator $\mathcal{L}$ in Theorem \ref{Theorem DPE 1} can be analytically calculated in this case as follows:
	\begin{align*}
	    \sup\limits_{u \in \controlset{U}}&\mathcal{L}^u \Phi(t,x,y)  = \\
	     & \sup\limits_{\alpha \in [-\pi,\pi]} \Big( \partial_t \Phi(t,x,y) + \big(1 - ay^2 + V_S\cos(\alpha)\big)\partial_x \Phi(t,x,y)   \\ 
	     & \qquad +  V_S \sin(\alpha) \partial_y \Phi(t,x,y) + \frac{1}{2}\sigma_x^2 \partial_{x}^2 \Phi(t,x,y) + \frac{1}{2}\sigma_y^2 \partial_{y}^2 \Phi(t,x,y) \Big).
	\end{align*}
	It can be shown that the controller value maximizing the above Dynkin operator is
	{
	\begin{align*}
	    \alpha^*(t,x,y) & := \operatorname*{arg\,max}_{\alpha \in [-\pi,\pi]} \Big( \cos(\alpha)\partial_x\Phi(t,x,y) + \sin(\alpha)\partial_y\Phi(t,x,y) \Big) \\
	    & = \arctan(\frac{\partial_y\Phi}{\partial_x\Phi})(t,x,y).
	\end{align*}}
	Therefore, the differential operator can be simplified to
	{
	\begin{align*}
	    \sup_{u \in \controlset{U}} \mathcal{L}^u & \Phi(t,x,y) = \partial_t \Phi(t,x,y) + (1 - ay^2)\partial_x \Phi(t,x,y) \\
	    & + \frac{1}{2}\sigma_x^2 \partial_{x}^2 \Phi(t,x,y) + \frac{1}{2}\sigma_y^2 \partial_{y}^2 \Phi(t,x,y) + V_S\|\nabla \Phi(t,x,y)\|,
	\end{align*}
	}
	where $\nabla\Phi(t,x,y) := \big[\partial_x\Phi(t,x,y) \quad \partial_y\Phi(t,x,y)\big]$.

\begin{figure}[t!]%
\centering
  \subfigure[The first scenario: the swimmer's speed is slower than the river current, the current being assumed uniform.]{\label{Fig:VF_1}\includegraphics[scale = 0.16]{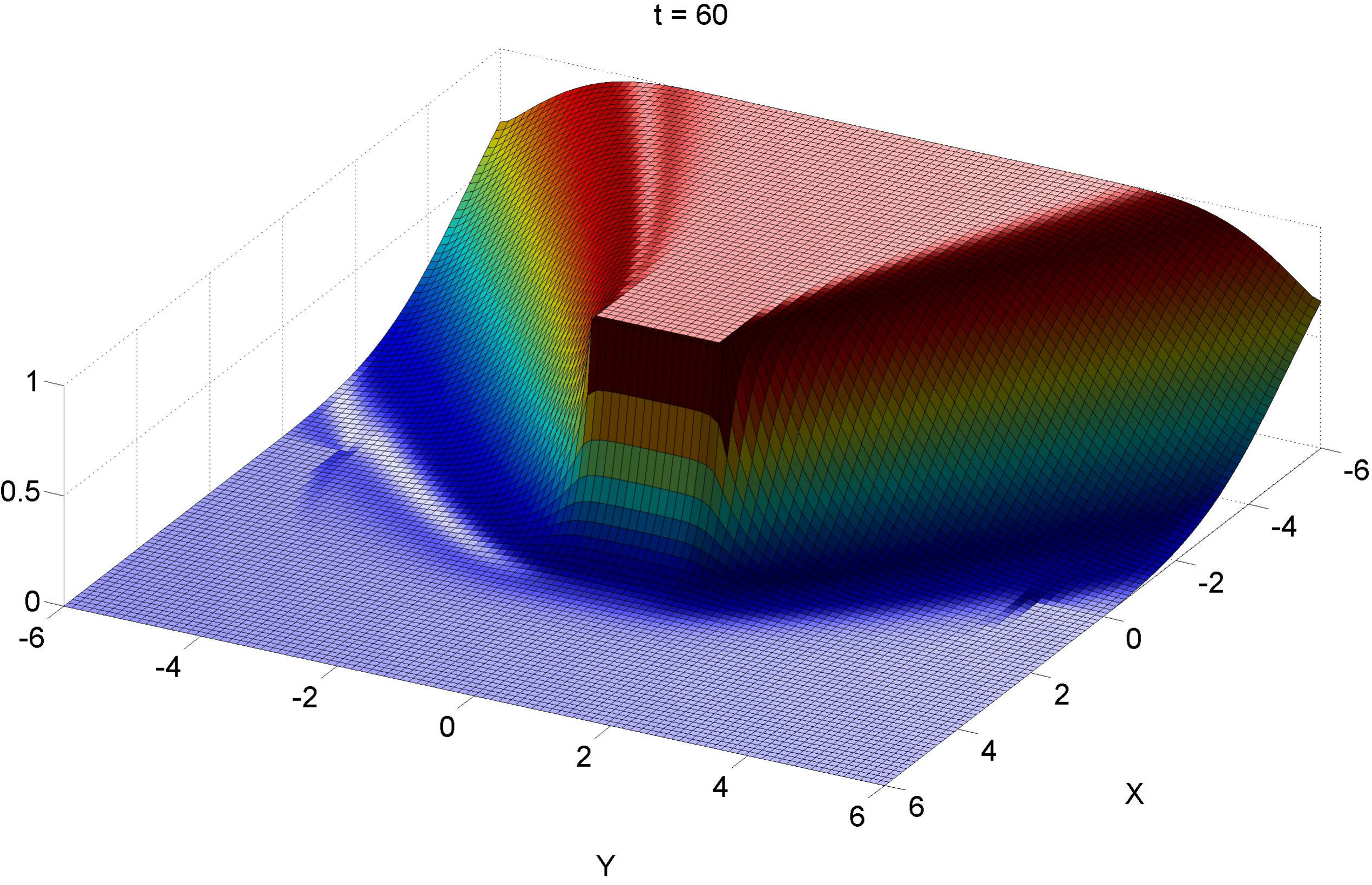}} \quad
  \subfigure[The second scenario: the swimmer's speed is slower than the maximum river current.]{\label{Fig:VF_2}\includegraphics[scale = 0.16]{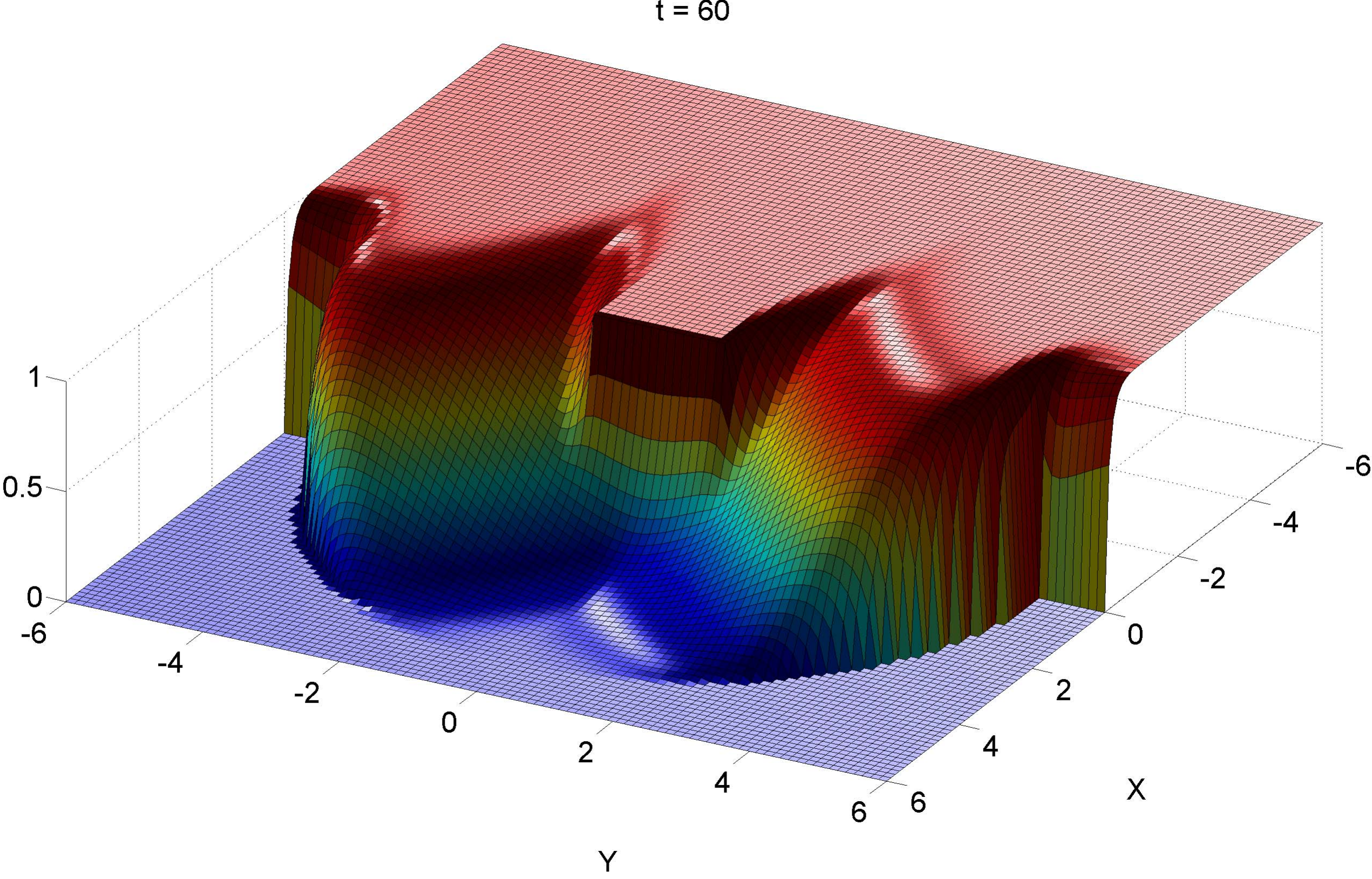}}\\
  \subfigure[The third scenario: the swimmer can swim faster than the maximum river current.]
        {\label{Fig:VF_3}\includegraphics[scale = 0.16]{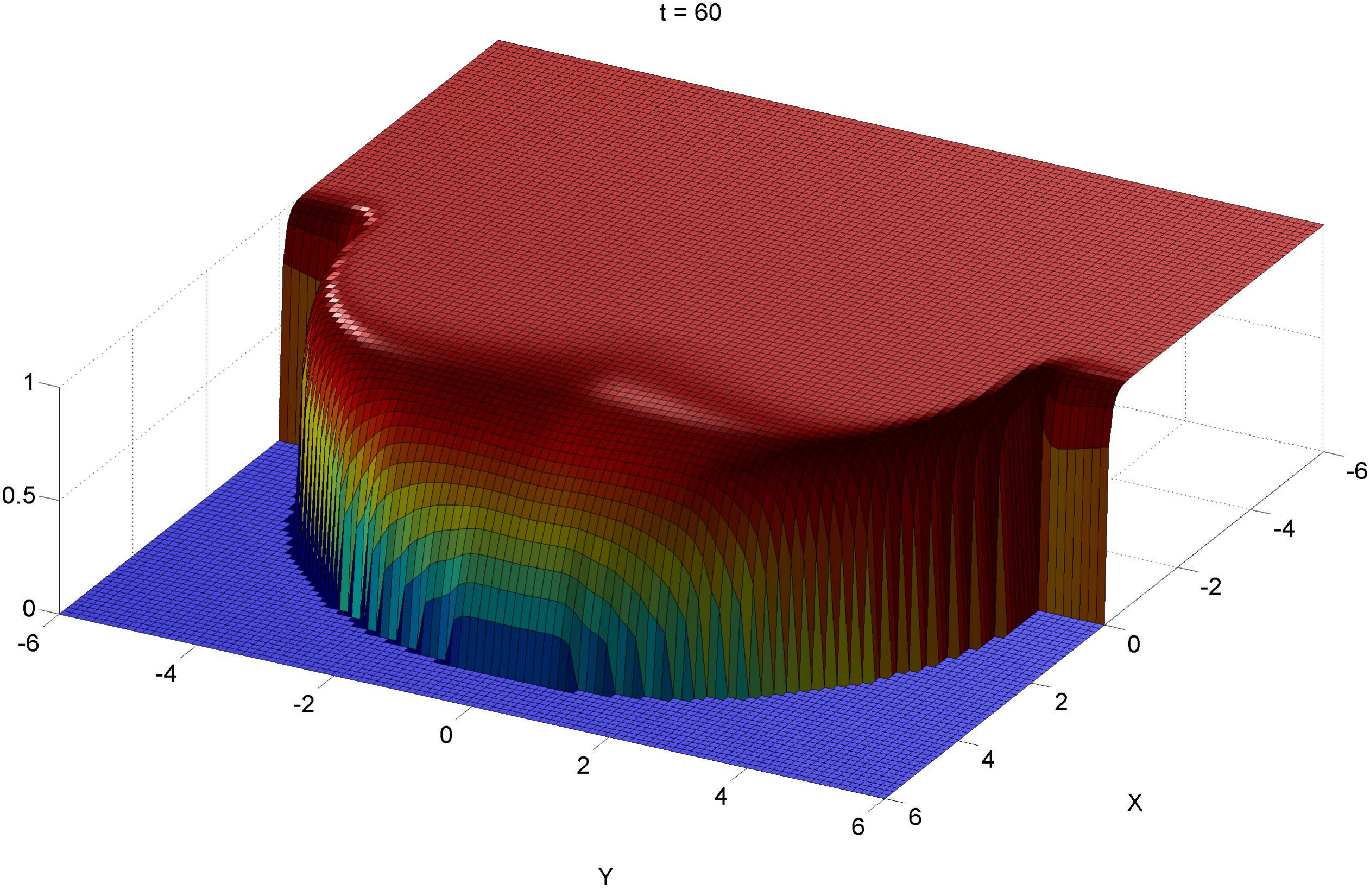}}%
  \caption{The value functions for the different scenarios}
\label{Fig:ValueFunc}
\end{figure}

\subsection{Simulation results}
	For the following numerical simulations we fix the diffusion coefficients $\sigma_x = 0.5$ and $\sigma_y = 0.2$. We investigate three different scenarios: first, we assume that the river current is uniform, i.e., $a = 0 \mathrm{m}^{-1}\mathrm{s}^{-1}$ in \eqref{river dynamic}. Moreover, we consider the case that the swimmer velocity is less than the current flow, e.g., $V_S = 0.6\; \mathrm{m}\mathrm{s}^{-1}$. Based on the above calculations, Figure \ref{Fig:VF_1} depicts the value function which is the numerical solution of the differential operator equation in Theorem \ref{Theorem DPE 1} with the corresponding terminal and lateral conditions. As expected, since the swimmer's speed is less than the river current, if he starts from the beyond the target he has less chance of reach the island. This scenario is also captured by the value function shown in Figure \ref{Fig:VF_1}.

\begin{figure}[t!]%
\centering
  \subfigure[The first scenario: the swimmer's speed is slower than the river current, the current being assumed uniform.]{\label{Fig:LS_1}\includegraphics[scale = 0.24]{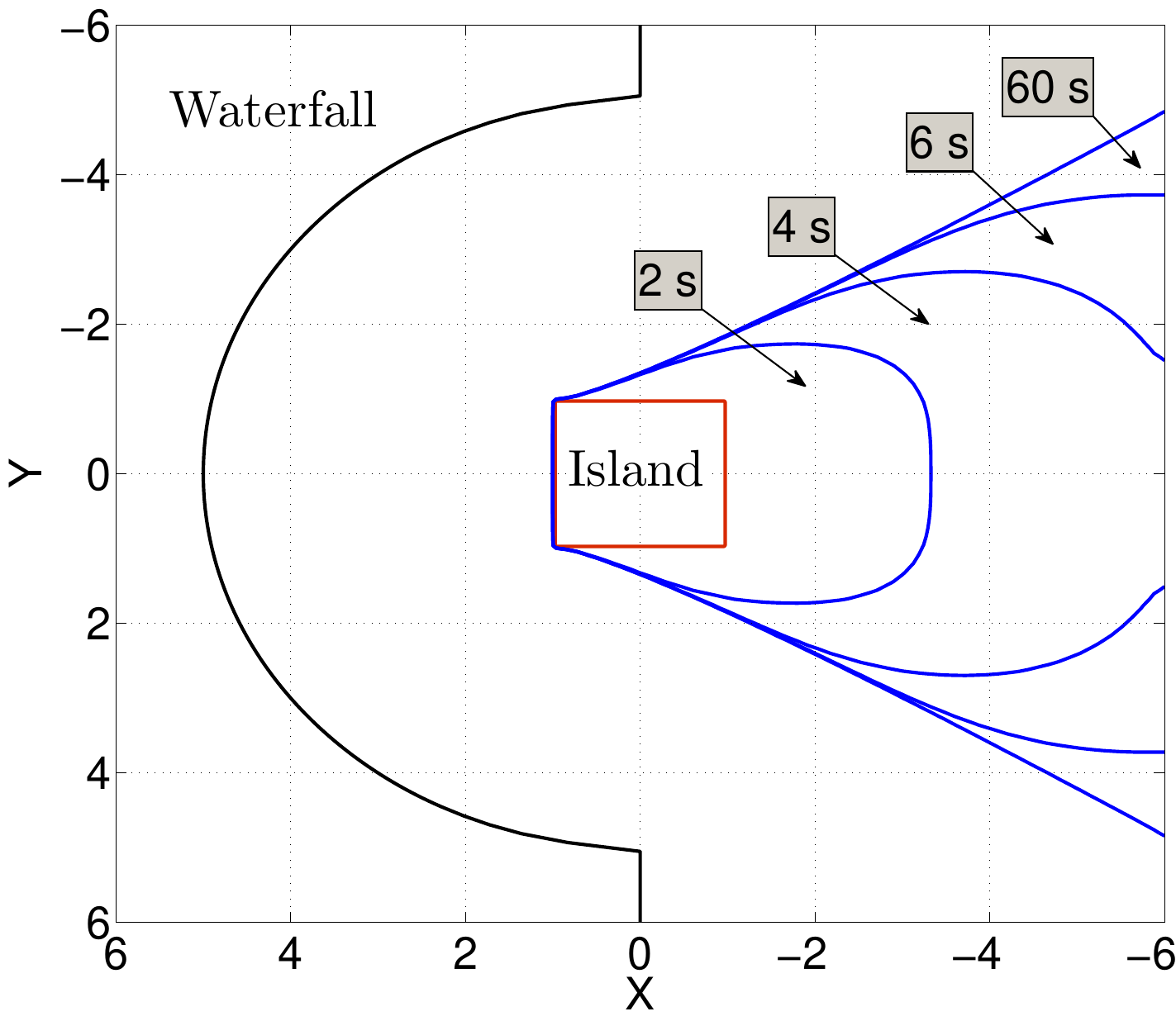}} \quad
  \subfigure[The second scenario: the swimmer's speed is slower than the maximum river current.]{\label{Fig:LS_2}\includegraphics[scale = 0.24]{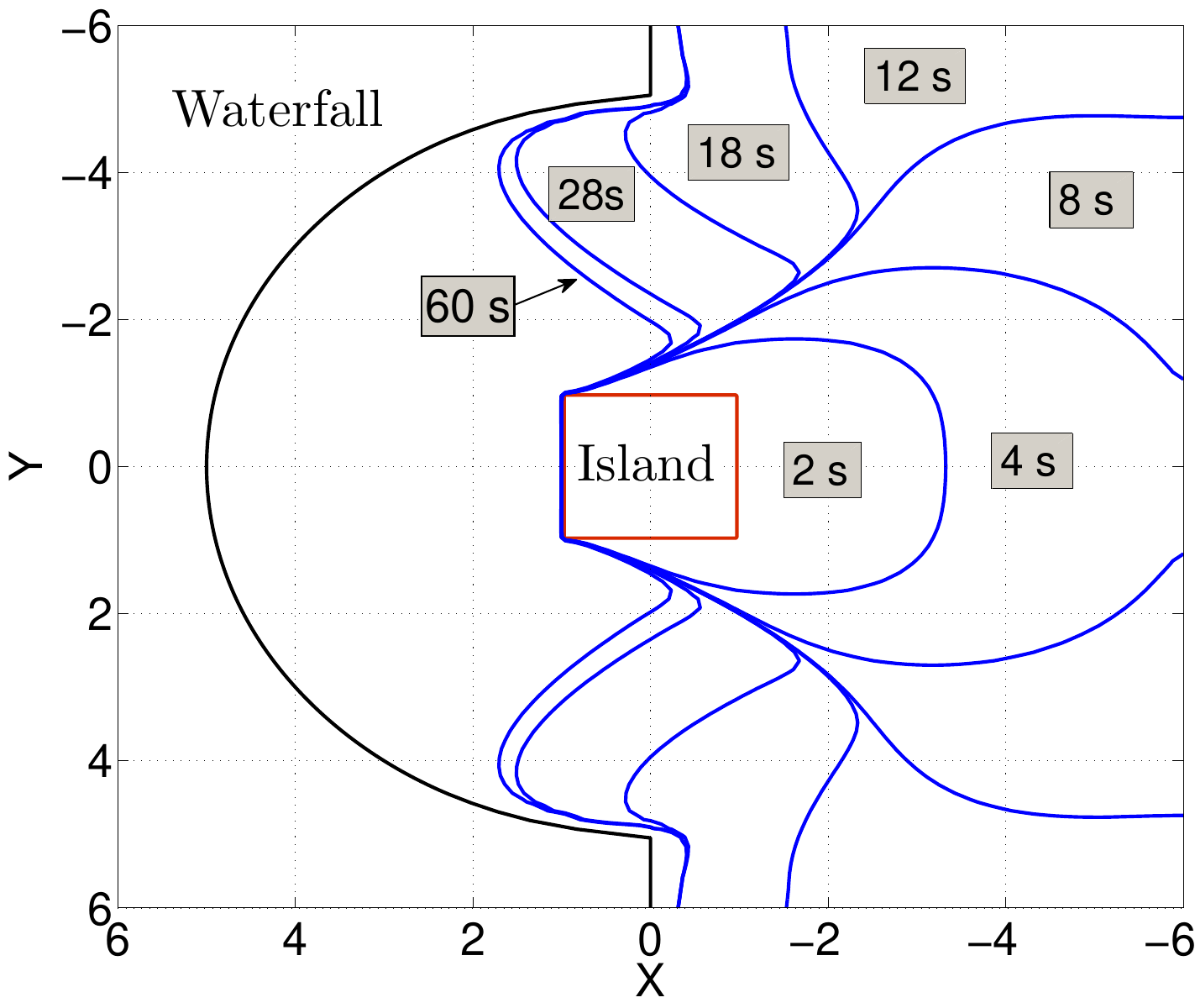}}\\
  \subfigure[The third scenario: the swimmer can swim faster than the maximum river current.]
        {\label{Fig:LS_3}\includegraphics[scale = 0.24]{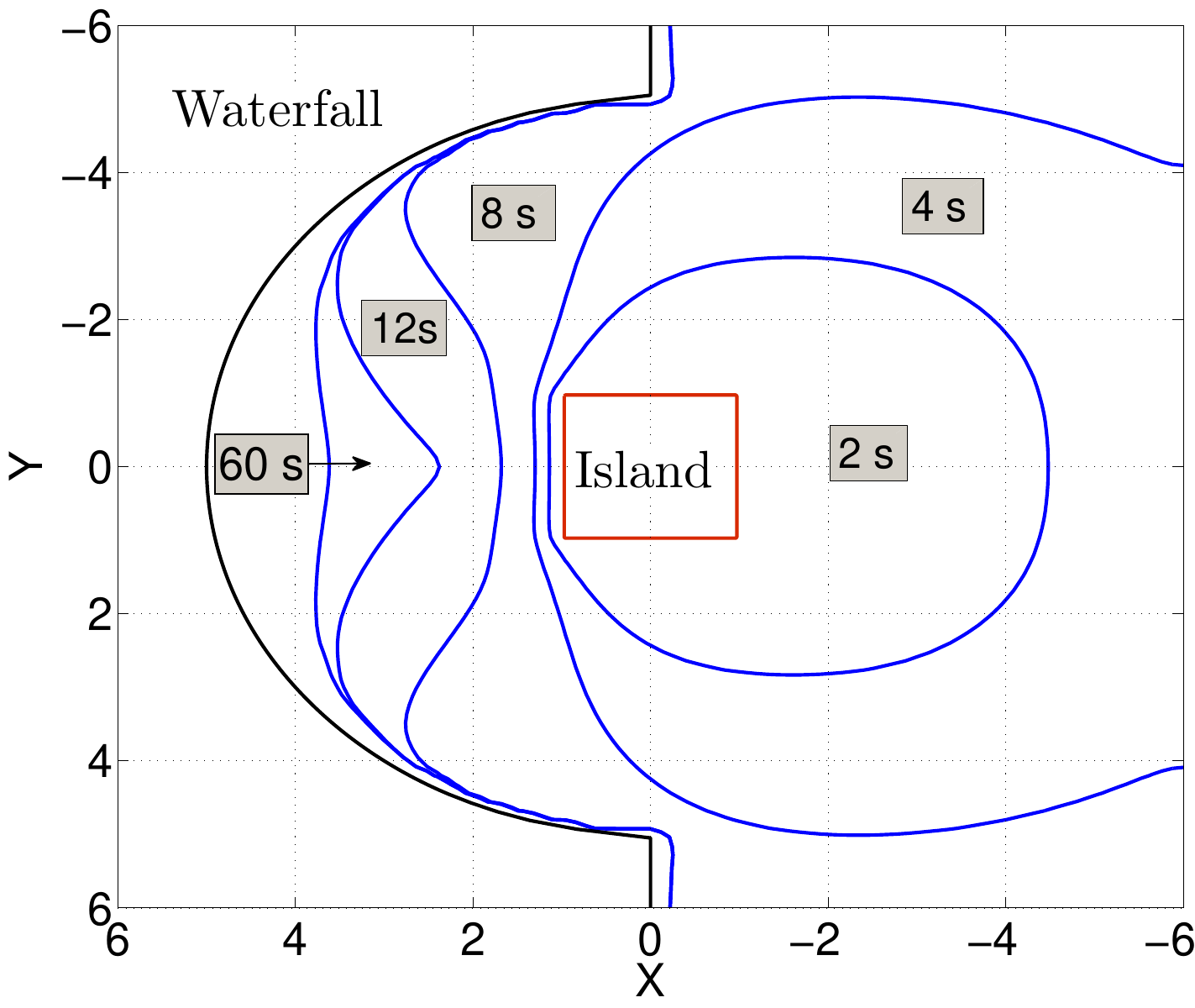}}%
  \caption{The level sets of the value functions for the different scenarios}
\label{Fig:levelsets}
\end{figure}

	Second, we assume that the river current is non-uniform and decreases with respect to the distance from the middle of the river. This means that the swimmer, even in the case that his speed is less than the current, has a non-zero probability of success if he initially swims to the sides of the river partially against its direction, followed by swimming in the direction of the current to reaches the target. This scenario is depicted in Figure \ref{Fig:VF_2}, where a non-uniform river current $a = 0.04 \mathrm{m}^{-1}\mathrm{s}^{-1}$ in \eqref{river dynamic} is considered.

	Third, we consider the case that the swimmer can swim faster than river current. In this case we expect the swimmer to succeed with some probability even if he starts from beyond the target. This scenario is captured in Figure \ref{Fig:VF_3}, where the reachable set (of course in probabilistic fashion) covers the entire navigation region of the river except the region near the waterfall.

	In the following we show the level sets of the aforementioned value functions for $p = 0.9$. As defined in  \secref{sec:connection} (and in particular in Proposition \ref{prop:Reach-V1}), these level sets, roughly speaking, correspond to the reachable sets with probability $p = 90\%$ in certain time horizons while the swimmer is avoiding the waterfall. By definition, as shown by the following figures, these sets are nested with respect to the time horizon.

	All simulations were obtained using the Level Set Method Toolbox \cite{ref:Mitchell-toolbox} (version 1.1), with a grid $101 \times 101$ in the region of simulation.
\section{Concluding Remarks and Future Direction}
\label{sec:conclusion}
	In this article we studied a class of stochastic reach-avoid problems from an optimal control perspective. The proposed framework provides a set characterization of the stochastic reach-avoid set based on discontinuous viscosity solutions of a second order PDE. In contrast to earlier approaches, this methodology is not restricted to almost-sure notions and allows for discontinuous payoff functions. We also provided theoretical justification to compute the desired reach-avoid set by means of off-the-shelf PDE solvers. 

	In future works we aim to extend our framework to stochastic motion-planning that indeed involves concatenating basic reachability maneuver studied in this work. Another extension to the current setting could be the existence of a second player who plays against our main objective, which is known as the stochastic differential game in literature. 

\section*{Acknowledgment}
	The authors are grateful to Ian Mitchell for his assistance and advice on the numerical coding of the examples. The authors thank V.\ S.\ Borkar, H.\ M.\ Soner, A.\ Ganguly, and S.\ Pal for helpful discussions and pointers to references.

\appendix
\section{Technical Proofs of \secref{sec:connection}}\label{app:A}

\begin{proof}[Proof of Proposition \ref{prop:Reach-V1}]
We first establish the equality of $V_1 = V_2$. To this end, let us fix $\control{u} \in \controlmaps{U}$ and $(t,x)$ in $\set S$. Observe that it suffices to show that pointwise on $\Omega$,
\begin{align*}
    \ind{A}(\traj{X}{t,x}{\control{u}}{\wh{\tau}}) = \Sup{t}{s}{T} \{\ind{A}(\traj{X}{t,x}{\control{u}}{s}) \mn \Inf{t}{r}{s} \ind{B^c}(\traj{X}{t,x}{\control{u}}{r}) \}.
\end{align*}
Since $A$ and $B$ are closed, thanks to Remark \ref{rem:entry time} one can see that
\begin{align*}
 \Sup{t}{s}{T} \{ &\ind{A}(\traj{X}{t,x}{\control{u}}{s}) \mn \Inf{t}{r}{s} \ind{B^c}(\traj{X}{t,x}{\control{u}}{r}) \} = 1\\
 & \Longleftrightarrow \exists s \in [t,T] ~ \traj{X}{t,x}{\control{u}}{s} \in A ~\text{and}~ \forall r \in [t,s] ~ \traj{X}{t,x}{\control{u}}{r} \in B^c\\
 & \Longleftrightarrow \exists s \in [t,T] ~ \tau_A \leq s \leq T ~\text{and}~ \tau_B > s\\
 & \Longleftrightarrow \traj{X}{t,x}{\control{u}}{\tau_A} = \traj{X}{t,x}{\control{u}}{\tau_A \mn \tau_B \mn T} = \traj{X}{t,x}{\control u} {\tau_{A \cup B} \mn T} \in A \\
 & \Longleftrightarrow \ind{A} \big(\traj{X}{t,x}{\control{u}}{\wh{\tau}}\big) = 1
\end{align*}
and since the functions take values in $\{0,1\}$, we have $V_1(t,x)=V_2(t,x)$.

As a first step towards proving $V_1 = V_3$, we start with establishing $V_3 \geq V_1$. It is straightforward from the definition that
\begin{align}
    \label{V3>V1}
    \begin{array}{l} 
    \sup\limits_{\tau \in \setofst{t,T}} \inf\limits_{\sigma \in \setofst{t,\tau}} \EE\Big[\ind{A} (\traj{X}{t,x}{\control{u}}{\tau}) \mn \ind{B^c}(\traj{X}{t,x}{\control{u}}{\sigma})\Big] \geq \inf\limits_{\sigma \in \setofst{t,\wh{\tau}}} \E{\ind{A}(\traj{X}{t,x}{\control{u}}{\wh{\tau}}) \mn \ind{B^c}(\traj{X}{t,x}{\control{u}}{\sigma})},
    \end{array}
\end{align}

where $\wh{\tau}$ is the stopping time defined in (\ref{V1}). For all stopping times $\sigma \in \setofst{t,\wh{\tau}}$, in view of (\ref{rem:exit time A}) we have
\begin{align*}
   \ind{B^c}(\traj{X}{t,x}{\control{u}}{\sigma}) = 0
   &\Lra  \traj{X}{t,x}{\control{u}}{\sigma} \in B \Lra \tau_B \le \sigma \le \wh{\tau} = \tau_A \mn \tau_B \mn T 
   \\&\Lra \tau_B = \sigma = \wh{\tau} < \tau_A \Lra \traj{X}{t,x}{\control{u}}{\wh{\tau}} \notin A \\
   &\Lra \ind{A}(\traj{X}{t,x}{\control{u}}{\wh{\tau}}) = 0
\end{align*}
This implies that for all $\sigma \in \setofst{t,\wh{\tau}}$,
\begin{equation*}
    \ind{A}(\traj{X}{t,x}{\control{u}}{\wh{\tau}}) \mn \ind{B^c}(\traj{X}{t,x}{\control{u}}{\sigma}) = \ind{A}(\traj{X}{t,x}{\control{u}}{\wh{\tau}}) \qquad \PP\text{-a.s.}
\end{equation*}
which, in connection with \eqref{V3>V1} leads to
\begin{align*}
    \sup_{\tau \in \setofst{t,T}} \inf_{\sigma \in \setofst{t,\tau}} \E{\ind{A}(\traj{X}{t,x}{\control{u}}{\tau}) \mn \ind{B^c}(\traj{X}{t,x}{\control{u}}{\sigma})} \geq
    \E{\ind{A}(\traj{X}{t,x}{\control{u}}{\wh{\tau}})}.
\end{align*}
By arbitrariness of the control strategy $\control{u}\in\controlmaps{U}$, we get $V_3 \ge V_1$. It remains to show  $V_2 \leq V_1$. Given $\control{u}\in \controlmaps{U}$ and $\tau \in \setofst{t,T}$, let us choose $\wh{\sigma} := \tau \mn \tau_B$. Note that since $t \le \wh{\sigma} \le \tau$ then $\wh \sigma \in \setofst{t,\tau}$. Hence,
\begin{align}
    \label{V3<V1}
    \inf_{\sigma \in \setofst{t,\tau}} \E{\ind{A}(\traj{X}{t,x}{\control{u}}{\tau}) \mn \ind{B^c}(\traj{X}{t,x}{\control{u}}{\sigma})} \le \E{\ind{A}(\traj{X}{t,x}{\control{u}}{\tau})\mn \ind{B^c}(\traj{X}{t,x}{\control{u}}{\wh{\sigma}})}.
\end{align}
Note that by an argument similar to the proof of Proposition \ref{prop:Reach-V1}, for all $\tau \in \setofst{t,T}$:
\begin{align*}
    \ind{A}(\traj{X}{t,x}{\control{u}}{\tau})\mn &\ind{B^c}(\traj{X}{t,x}{\control{u}}{\wh{\sigma}}) = 1  \Lra   \traj{X}{t,x}{\control{u}}{\tau}\in A ~\text{and}~ \traj{X}{t,x}{\control{u}}{\wh{\sigma}} \notin B\\
    & \Lra \tau_A \le \tau \le T ~\text{and}~ \wh{\sigma} \neq \tau_B\\
    & \Lra \tau_A \le \tau \le T ~\text{and}~ \tau_A \le \wh{\sigma} = \tau < \tau_B \\
    & \Lra \wh{\tau} = \tau_A \mn \tau_B \mn T = \tau_A \Lra \ind{A}(\traj{X}{t,x}{\control{u}}{\wh{\tau}}) = 1.
\end{align*}
It follows that for all $\tau \in \setofst{t,\tau}$,
\begin{equation*}
     \ind{A}(\traj{X}{t,x}{\control{u}}{\tau})\mn \ind{B^c}(\traj{X}{t,x}{\control{u}}{\wh{\sigma}}) \le \ind{A}(\traj{X}{t,x}{\control{u}}{\wh{\tau}}) \qquad \PP\text{-a.s.}
\end{equation*}
which in connection with \eqref{V3<V1} leads to
\begin{align*}
    \sup_{\tau \in \setofst{t,T}} \inf_{\sigma \in \setofst{t,\tau}} \E{\ind{A}(\traj{X}{t,x}{\control{u}}{\tau}) \mn \ind{B^c}(\traj{X}{t,x}{\control{u}}{\sigma})} \leq
    \E{\ind{A}(\traj{X}{t,x}{\control{u}}{\wh{\tau}})}.
\end{align*}
By arbitrariness of the control strategy $\control{u}\in \controlmaps{U}$ we arrive at $V_3 \le V_1$.

We now show the second assertion. since $A$ is closed, making use of the implication \eqref{rem:exit time A} and the definition of reach-avoid set in \ref{def:RA within}, we can express the set $\RA(t,p;A,B)$ by
	\begin{align}
	\label{Reach-Avoid equivalent def}
		\RA(t,p;A,B) = \Big \{x \in & \R^n ~\big | ~ \exists \control{u}  \in \controlmaps{U} ~:~ \notag \\
		& \PP\big( \tau_A < \tau_B ~~ \text{and} ~~ \tau_A \leq T \big) > p \Big\}. 
	\end{align}
Also, in view of the properties \eqref{rem:A cup B} and \eqref{rem:exit time closed set}, for any control $\control{u} \in \controlmaps{U}$ we have
	\begin{align*}
		\traj{X}{t,x}{\control{u}}{\wh{\tau}} \in A \Lra \tau_A \leq \wh{\tau} ~\text{and}~ \wh{\tau}\neq \tau_B \Lra T \geq \wh{\tau} = \tau_A < \tau_B,
	\end{align*}
indicating that the sample path $\traj{X}{t,x}{\control{u}}{\cdot}$ hits the set $A$ before $B$ at the time $\wh{\tau} \leq T$. Moreover,
	\begin{align*}
		\traj{X}{t,x}{\control{u}}{\wh{\tau}} \notin A \Lra \wh{\tau} \neq \tau_A \Lra \wh{\tau} = (\tau_B \mn T) < \tau_A,
	\end{align*}
and this means that the sample path does not succeed in reaching $A$ while avoiding set $B$ within time $T$. Therefore, the event $ \{ \tau_A < \tau_B ~~ \text{and} ~~ \tau_A \leq T\}$ is equivalent to $ \{\traj{X}{t,x}{\control{u}}{\wh{\tau}} \in A \}$, and
	\begin{align*}
		\PP\big( \tau_A < \tau_B ~~ \text{and} ~~ \tau_A \leq T \big) = \E{\ind{A}( \traj{X}{t,x}{\control{u}}{\wh{\tau}})}.
	\end{align*}
This, in view of \eqref{Reach-Avoid equivalent def} and arbitrariness of control strategy $\control{u}\in\controlmaps{U}$ leads to the desired assertion.
\end{proof}

\section{Technical proofs of \secref{sec:Alternative Characterization}}
\label{app:B}
\begin{proof}[Proof of Proposition \ref{prop:J1 semicontinuous}]
We first prove continuity of $\wh{\tau}(t,x)$ with respect to $(t,x)$. 
    Let us take a sequence $(t_n,x_n) \ra (t_0,x_0)$, and let $\big(\traj{X}{t_n,x_n}{\control{u}}{r}\big)_{r \ge t_n}$ be the solution of \eqref{SDE} for a given policy $\control{u} \in \controlmaps{U}$. Let us recall that by definition we assume that $\traj{X}{t,x}{\control{u}}{s} \Let x$ for all $s \in [0, t]$. Here we assume that $t_n \le t$, but one can effectively follow the same technique for $t_n > t$. Notice that it is straightforward to observe that by the definition of stochastic integral in \eqref{SDE} we have
    \begin{equation*}
    	\traj{X}{t_n,x_n}{\control{u}}{r} = \traj{X}{t_n,x_n}{\control{u}}{t} + \int_{t}^r f\big( \traj{X}{t_n,x_n}{\control{u}}{s}, u_s \big) \diff s + \int_{t}^r \sigma \big(\traj{X}{t_n,x_n}{\control{u}}{s}, u_s\big) \diff W_s 
    \end{equation*}
    Therefore, by virtue of \cite[Theorem 2.5.9, p.\ 83]{Krylov_ControlledDiffusionProcesses}, for all $q \ge 1$ we have 
    \begin{align*}
    	\EE \Big[\sup_{r \in [t,T]}& \big \| \traj{X}{t,x}{\control{u}}{r}  -   \traj{X}{t_n,x_n}{\control{u}}{r}   \big \|^{2q} \Big]  \le C_1(q,T,K) \EE \Big[ \big \| x - \traj{X}{t_n,x_n}{\control{u}}{t}   \big \|^{2q} \Big]\\
    	& \le 2^{2q-1}C_1(q,T,K) \EE \Big[ \|x-x_n\|^{2q} + \big \| x_n - \traj{X}{t_n,x_n}{\control{u}}{t}   \big \|^{2q} \Big],
    \end{align*}   
    where in light of \cite[Corollary 2.5.12, p.\ 86]{Krylov_ControlledDiffusionProcesses}, it leads to 
    \begin{align}
    \label{ineq-init}
    	    	\EE \Big[\sup_{r \in [t,T]} \big \|  \traj{X}{t,x}{\control{u}}{r}  -& \traj{X}{t_n,x_n}{\control{u}}{r}   \big \|^{2q} \Big] \le \\
    	    	& \notag C_2(q,T,K,\|x\|)\big(\|x-x_n\|^{2q} + |t-t_n|^q\big).
    \end{align}
	In the above relations $K$ is the Lipschitz constant of $f$ and $\sigma$; $C_1$ and $C_2$ are constant depending on the indicated parameters. Hence, in view of Kolmogorov's continuity criterion \cite[Corollary 1 Chap.\ IV, p.\ 220]{ref:Protter-2005}, one may consider a version of the stochastic process $\traj{X}{t,x}{\control{u}}{\cdot}$ which is continuous in $(t,x)$ in the topology of uniform convergence on compacts. This yields to the fact that $\PP$-a.s, for any $\eps > 0 $, for all sufficiently large $n$,
    \begin{equation}\label{Tube}
        \traj{X}{t_n,x_n}{\control{u}}{r} \in \ball{B}_{\eps}\big(\traj{X}{t_0,x_0}{\control{u}}{r}\big), \qquad \forall r \in [t_n,T],
    \end{equation}
    where $\ball{B}_\eps(y)$ denotes the ball centered at $y$ and radius $\eps$. Based on the Assumptions \ref{a:exit time}.\ref{a:exit time:degenerate} and \ref{a:exit time}.\ref{a:exit time:Set condition}, it is a well-known property of non-degenerate processes that the set of sample paths that hit the boundary of $O$ and do not enter the set is negligible \cite[Corollary 3.2, p.\ 65]{ref:Bass-1998}. Hence, by the definition of $\wh \tau$ and \eqref{rem:exit time A}, one can conclude that 
    \begin{equation*}
        \forall \delta >0, ~ \exists \eps>0, \quad \bigcup_{s \in [t_0,\wh \tau(t_0,x_0) - \delta]} \ball{B}_{\eps}(\traj{X}{t_0,x_0}{\control{u}}{s}) \cap \ol{O} = \emptyset \qquad \PP \text{-a.s.}
    \end{equation*}
    This together with \eqref{Tube} indicates that $\PP$-a.s.\ for all sufficiently large $n$,
    \begin{equation*}
    	\traj{X}{t_n,x_n}{\control{u}}{r} \notin \ol{O}, \qquad \forall r \in [t_n, \wh \tau(t_0,x_0)[ ~,
    \end{equation*}    
    which in conjunction with $\PP$-a.s.\ continuity of sample paths immediately leads to
    \begin{equation}
	\label{liminf tau_n}
        \liminf_{(t_n,x_n) \ra (t,x)} \wh \tau(t_n,x_n) \ge \wh \tau(t_0,x_0) \qquad \PP \text{-a.s.}
    \end{equation}
    On the other hand by the definition of $\wh \tau$ and Assumptions \ref{a:exit time}.\ref{a:exit time:degenerate} and \ref{a:exit time}.\ref{a:exit time:Set condition}, again in view of \cite[Corollary 3.2, p.\ 65]{ref:Bass-1998},
    \begin{equation*}
        \forall \delta>0, \quad \exists s \in [\tau_O(t_0,x_0), \tau_O(t_0,x_0)+\delta[, \quad  \traj{X}{t_0,x_0}{\control{u}}{s} \in O^\circ \quad \PP \text{-a.s.},
    \end{equation*}
    where $\tau_O$ is the first entry time to $O$, and $O^\circ$ denotes the interior of the set $O$. Hence, in light of \eqref{Tube}, $\PP$-a.s.\ there exists $\eps >0 $, possibly depending on $\delta$, such that for all sufficiently large $n$ we have $\traj{X}{t_n,x_n}{\control{u}}{s} \in \ball{B}_{\eps}(\traj{X}{t_0,x_0}{\control{u}}{s}) \subset O$.
    According to the definition of $\tau_O(t_n,x_n)$ and \eqref{rem:exit time A}, this implies $\tau_O(t_n,x_n) \le s < \tau_O(t_0,x_0)+\delta$. From arbitrariness of $\delta$ and the definition of $\wh \tau$ in \eqref{J}, it leads to
     \begin{equation*}\label{limsup tau_n}
        \limsup_{(t_n,x_n) \ra (t,x)} \wh \tau(t_n,x_n) \le \wh \tau(t_0,x_0) \qquad \PP \text{-a.s.},
    \end{equation*}
    where in conjunction with \eqref{liminf tau_n}, $\PP$-a.s.\ continuity of the map $(t,x) \mapsto \wh \tau(t,x)$ at $(t_0,x_0)$ follows.

   It remains to show lower semicontinuity of $J$. Note that $J$ is bounded since $\ell$ is. In accordance with the $\PP$-a.s.\ continuity of $\traj{X}{t,x}{\control{u}}{r}$ and $\wh \tau(t,x)$ with respect to $(t,x)$, and Fatou's lemma, we have
   \begin{align}
        \liminf_{n \ra \infty}  & J\big(t_n, x_n,\control{u}\big)
         = \liminf_{n \ra \infty} \E{ \el{\traj{X}{t_n,x_n}{\control{u}}{\wh{\tau}(t_n,x_n)} } } \nonumber \\
        & = \liminf_{n \ra \infty} \E{ \el{ \traj{X}{t_n,x_n}{\control{u}}{\wh{\tau}(t_n,x_n)}
        -\traj{X}{t,x}{\control{u}}{\wh{\tau}(t_n,x_n)} + \traj{X}{t,x}{\control{u}}{\wh{\tau}(t_n,x_n)}
        - \traj{X}{t,x}{\control{u}}{\wh{\tau}(t,x)} +\traj{X}{t,x}{\control{u}}{\wh{\tau}(t,x)}} } \nonumber \\
        & \label{fatou} = \liminf_{n \ra \infty} \E{ \el{\eps_n + \traj{X}{t,x}{\control{u}}{\wh{\tau}(t,x)}} }
        \ge \E{ \liminf_{n \ra \infty} \el{\eps_n + \traj{X}{t,x}{\control{u}}{\wh{\tau}(t,x)}}}\\
        &\ge 
         \E { \el{\traj{X}{t,x}{\control{u}}{\wh{\tau}(t,x)}} } = J(t,x,\control{u}) \nonumber,
    \end{align}
	where inequality in \eqref{fatou} follows from Fatou's Lemma, and $\eps_n \ra 0 $ $\PP$-a.s.\ as $n$ tends to $\infty$. Note that by definition $\traj{X}{t,x}{\control{u}}{\wh \tau (t_n,x_n)} = x$ on the set $\{ \wh \tau(t_n,x_n) < t \} $.
\end{proof}

\begin{proof}[Proof of Proposition \ref{prop:unif cont}]
	Let us consider a version of $\traj{X}{t,x}{\control{u}}{\cdot}$ which is almost surely continuous in $(t,x)$ uniformly respect to the policy $\control{u}$; this is always possible since the constant $C_2$ in \eqref{ineq-init} does not depend on $\control{u}$. That is, $\control{u}$ may only affect a negligible subset of $\Omega$; we refer to \cite[Theorem 72 Chap. IV, p.\ 218]{ref:Protter-2005} for further details on this issue. Hence, all the relations in the proof of Proposition \ref{prop:J1 semicontinuous}, in particular \eqref{Tube}, hold if we permit the control policy $\control{u}$ to depend on $n$ in an arbitrary way. Therefore, the assertions of Proposition \ref{prop:J1 semicontinuous} holds uniformly with respect to $(\control{u}_n)_{n \in \N} \subset \controlmaps{U}$. That is, for all $(t,x) \in \set{S}$, $(t_n,x_n) \ra (t,x)$, and $(\control{u}_n)_{n \in \N}$, with probability one we have 
	\begin{align}
		\label{limits}
		\left\{
		\begin{array}{l} \vspace{1mm}
		\lim\limits_{n \ra \infty} \sup\limits_{s \in [0,T]} \big\| \traj{X}{t_n,x_n}{\control{u}_n}{s} - \traj{X}{t,x}{\control{u}_n}{s} \big\| = 0,\\ 
		\lim\limits_{n \ra \infty} \big | \wh{\tau}(t_n,x_n) - \wh{\tau}(t,x) \big | = 0
		\end{array}
		\right.
	\end{align}
	where $\wh{\tau}$ is as defined in \eqref{V} while the solution process is driven by control policies $\control{u}_n$. Moreover, according to \cite[Corollary 2.5.10, p.\ 85]{Krylov_ControlledDiffusionProcesses} for every $r,s \in [t,T]$ and $q\ge 1$ we have
		\begin{equation*}
			\EE \Big[\big \| \traj{X}{t,x}{\control{u}}{r} - \traj{X}{t,x}{\control{u}}{s} \big \|^{2q} \Big] \le C_3 \big( q,T,K, \|x\| \big) \big | r - s \big | ^{q}, 
		\end{equation*} 
	following the arguments in the proof of Proposition \ref{prop:J1 semicontinuous} in conjunction with above inequality, one can also deduce that the mapping $s \mapsto \traj{X}{t,x}{\control{u}}{s}$ is $\PP$-a.s. continuous uniformly with respect to $\control{u}$.  Hence, one can infer that for all $(t,x) \in \set{S}$, with probability one we have 
	\begin{align*}
		\lim_{n \ra \infty} \big\| \traj{X}{t_n,x_n}{\control{u}_n}{\wh{\tau}(t_n,x_n)} - \traj{X}{t,x}{\control{u}_n}{\wh{\tau}(t,x)} \big\| \le \lim_{n \ra \infty} \big\| & \traj{X}{t_n,x_n}{\control{u}_n}{\wh{\tau}(t_n,x_n)}  - \traj{X}{t,x}{\control{u}_n}{\wh{\tau}(t_n,x_n)} \big\| + \lim_{n \ra \infty} \big\| \traj{X}{t,x}{\control{u}_n}{\wh{\tau}(t_n,x_n)} - \traj{X}{t,x}{\control{u}_n}{\wh{\tau}(t,x)} \big\| = 0.
	\end{align*}
	Notice that the first limit term above tends to zero as the version of the solution process $\traj{X}{t,x}{\control{u_n}}{\cdot}$ on the compact set $[0,T]$ is continuous in the initial condition $(t,x)$ uniformly with respect to $n$. The second term is the consequence of limits in \eqref{limits} and continuity of the mapping $s \mapsto \traj{X}{t,x}{\control{u}_n}{s}$ uniformly in $n \in \N$. 
\end{proof}
\bibliographystyle{amsalpha} 
\bibliography{ref,ref_MohajerinEsfahani}

\end{document}